\documentclass{amsart}
\usepackage{graphicx}

\def\rr{\mathbb R}

\def\HH{\mathcal H}
\def\AA{\mathcal A}

\newcommand {\nc}   {\newcommand}
\nc {\be}   {\begin{equation}} \nc {\ee}   {\end{equation}} \nc
{\beq}  {\begin{eqnarray}} \nc {\eeq}  {\end{eqnarray}} \nc {\beqs}
{\begin{eqnarray*}} \nc {\eeqs} {\end{eqnarray*}}
\def\edc{\end{document}}

\newtheorem{thm}{Theorem}[section]

\newtheorem{lem}[thm]{Lemma}

\newtheorem{prop}[thm]{Proposition}
\newtheorem{rk}[thm]{Remark}

\numberwithin{equation}{section}

\def\intspace{{\displaystyle\int_0^L}}
\def\inttime{{\displaystyle\int_0^{+\infty}}}

\theoremstyle{definition}

\numberwithin{equation}{section}
\begin{document}
	
	\title[Stability results of a distributed problem involving  Bresse system]{Stability results of a distributed problem involving  Bresse system with history and/or Cattaneo law under fully Dirichlet or mixed  boundary conditions}

	\author{Farah Abdallah}
		\address{Universit\'e Libanaise\\
			Facult\'e des Sciences 1\\
			EDST, Equipe EDP-AN\\
			Hadath, Beyrouth, Liban}
		\email{fara7abdallah@gmail.com}
	
	\author{Mouhammad Ghader}
	\address{Universit\'e Libanaise\\
		Ecole Doctorale des Sciences \\
		et de Technologie, Equipe EDP-AN\\
		Hadath, Beyrouth, Liban}
	\email{mhammadghader@hotmail.com}
	
	\author{Ali Wehbe}
	\address{Universit\'e Libanaise\\
		Facult\'e des Sciences 1\\
		EDST, Equipe EDP-AN\\
		Hadath, Beyrouth, Liban}
	\email{ali.wehbe@ul.edu.lb}

	\date{}
	

	\begin{abstract}
		In this paper, we study the stability of  a one-dimensional Bresse system with infinite memory type control  
		and/or with heat conduction given by Cattaneo's law acting in the shear angle displacement. When the thermal effect vanishes, the system becomes elastic with  memory term acting on one equation. Unlike \cite{Lucifatori}, \cite{hugo-Sare-racke}, and \cite{santossoufyane}, we consider the interesting case of fully Dirichlet boundary conditions. Indeed, under equal speed of propagation condition, we establish the exponential stability of the system.
		However, in the natural physical case when the speeds of propagation are different, using a spectrum method, we show that the Bresse system is not uniformly stable. In this case, we establish a polynomial energy decay rate. Our study is valid for all other mixed boundary conditions and generalizes that of
		\cite{Lucifatori}, \cite{hugo-Sare-racke}, and \cite{santossoufyane}.
	\end{abstract}
	
	\maketitle
	
	\section{Introduction}\label{se1}
	In this paper, we study the stability of the Bresse system with history and/or heat conduction given by Cattaneo's law. This system  defined on $\left(0,L\right)\times\left(0,+\infty\right)$ takes the following  form
	\begin{equation}\label{eqq1.1'}
	\left\{
	\begin{array}{lll}
	\displaystyle{
		\rho_1\varphi_{tt}-k_1 \left(\varphi_x+\psi+\mathrm{l} w\right)_x-\mathrm{l} k_3\left(w_x-\mathrm{l}\varphi\right)=0,}\\ \\
	\displaystyle{\rho_2 \psi_{tt}-k_2\psi_{xx}+k_1\left(\varphi_x+\psi+\mathrm{l}w\right)-\int_0^{+\infty} g\left(s\right)\psi_{xx}\left(x,t-s\right)ds+\delta\theta_x=0,}\\ \\
	\displaystyle{\rho_1w_{tt}-k_3\left(w_x-\mathrm{l}\varphi\right)_x+\mathrm{l}k_1\left(\varphi_x+\psi+\mathrm{l}w\right)=0,}
	\\ \\
	\displaystyle{\rho_3 \theta_t+q_x+\delta \psi_{tx}=0,}\\ \\
	\displaystyle{\tau q_t+\beta q+\theta_x=0,}
	\end{array}
	\right.
	\end{equation}
	with fully Dirichlet boundary conditions
	\begin{equation}
	\begin{array}{lll}\label{DDDD}
	\varphi\left(0,\cdot\right)=\varphi\left(L,\cdot\right)=\psi\left(0,\cdot\right)=\psi\left(L,\cdot\right)=0\quad&\text{in }\mathbb{R}_{+},\\
	w\left(0,\cdot\right)=w\left(L,\cdot\right)=\theta\left(0,\cdot\right)=\theta\left(L,\cdot\right)=0\quad&\text{in }\mathbb{R}_{+},
	\end{array}
	\end{equation}
	or with Dirichlet-Neumann-Dirichlet-Dirichlet boundary conditions
	\begin{equation}\label{DNDD}
	\begin{array}{lll}
	\varphi\left(0,\cdot\right)=\varphi\left(L,\cdot\right)=\psi_x\left(0,\cdot\right)=\psi_x\left(L,\cdot\right)=0\quad&\text{in }\mathbb{R}_{+},\\
	w\left(0,\cdot\right)=w\left(L,\cdot\right)=\theta\left(0,\cdot\right)=\theta\left(L,\cdot\right)=0\quad&\text{in }\mathbb{R}_{+},
	\end{array}
	\end{equation}
	or  with Dirichlet-Neumann-Neumann-Dirichlet boundary conditions
	\begin{equation}\label{DNND}
	\begin{array}{lll}
	\varphi\left(0,\cdot\right)=\varphi\left(L,\cdot\right)=\psi_x\left(0,\cdot\right)=\psi_x\left(L,\cdot\right)=0\quad&\text{in }\mathbb{R}_{+},\\
	w_x\left(0,\cdot\right)=w_x\left(L,\cdot\right)=\theta\left(0,\cdot\right)=\theta\left(L,\cdot\right)=0\quad&\text{in }\mathbb{R}_{+},
	\end{array}
	\end{equation}
	in addition to  the following initial conditions
	$$
	\begin{array}{lll}
	\varphi\left(\cdot,0\right)=\varphi_0\left(\cdot\right),\ \psi\left(\cdot,-t\right)=\psi_0\left(\cdot,t\right),\
	
	w\left(\cdot,0\right)=w_0\left(\cdot\right),\ \theta(0,\cdot)=\theta_0,\ q(0,\cdot)=q_0,\\
	
	\varphi_t\left(\cdot,0\right)=\varphi_1\left(\cdot\right),\
	\psi_t\left(\cdot,0\right)=\psi_1\left(\cdot\right),\
	w_t\left(\cdot,0\right)=w_1\left(\cdot\right) \ \ \ \textrm{in} \ (0,L).\\
	\end{array}
	$$
	
	The functions $\varphi,\ \psi,$ and  $w$     model the vertical, shear angle, and   longitudinal  displacements  of the filament. The functions  $\theta$ and $q$ model  the temperature difference and  the heat flux respectively.
	The coefficients $\rho_1,\ \rho_2,\ \rho_3,\ k_1,\ k_2,\ k_3,\ \mathrm{l},\ \delta,\ \tau, \ \beta$   are positive constants. The integral term represents a history term with kernel $g$ satisfying the following hypothesis:
	\begin{itemize}
		\item [{\rm (H)}]  $g:\mathbb{R}_+\to \mathbb{R}_+$ is a non-increasing differentiable  function   such that  $\displaystyle{\lim_{s\to0} g\left(s\right)}$  exists and there exists $c>0$ such that $$g'\left(s\right)\leq -c g\left(s\right).$$ Furthermore, we assume that $\widetilde{k}_2>0$ where $\widetilde{k}_{2}:=k_2- g_0,\ \text{ and } g_0=\displaystyle{\int_{0}^{+\infty}g\left(s\right)ds}.$
	\end{itemize}
	Indeed, the condition $\displaystyle{\lim_{s\to0} g\left(s\right)}$  is sufficient near zero and can replace the condition
	$g'\left(s\right)\geq -\widetilde{c} g\left(s\right)$ for some $\widetilde{c}>0$ considered in
	\cite{Lucifatori}, \cite{hugo-Sare-racke},  and \cite{santossoufyane}. Therefore, in this paper,
	hypothesis {\rm (H)} is an improved   condition  on the kernel function appearing in the history term.\\
	
	When $\delta=0$, decoupling occurs and  the  thermal effect in system \eqref{eqq1.1'} vanishes. Consequently, the study of the stability of system  \eqref{eqq1.1'} is reduced to the study of
	the Bresse system without heat conduction but with an infinite memory type control  acting only in the shear angle displacement. The Bresse system is usually considered in studying elastic structures of the arcs type (see \cite{JGJ}). It can be expressed by the equations of motion
	$$
	\begin{array}{rcl}
	\rho_1\varphi_{tt} & = & Q_x +\mathrm{l}N\\
	\rho_2 \psi_{tt} & = & M_x-Q \\
	\rho_1w_{tt} & = & N_x-\mathrm{l}Q
	\end{array}
	$$
	where
	$$
	\begin{array}{rcl}
	N & = & \displaystyle{k_3\left(w_x-\mathrm{l}\varphi\right)}\\
	Q & = & \displaystyle{k_1\left(\varphi_x+\psi+\mathrm{l}w\right)}\\
	M & = & \displaystyle{k_2\psi_{x}-\int_0^{+\infty} g\left(s\right)\psi_{x}\left(x,t-s\right)ds}
	\end{array}
	$$
	are the stress strain relations for elastic behavior.  Here $\rho_1=\rho A, \ \rho_2=\rho I, \ k_1=k'GA, \ k_3=EA, \ k_2=EI, \ \mathrm{l}=R^{-1}$ where $\rho$ is the density of the material, $E$ is the modulus of elasticity, $G$ is the shear modulus, $k'$ is the shear factor, $A$ is the cross-sectional area, $I$ is the second moment of area of the cross-section, and $R$ is the radius of curvature. $\varphi, \ \psi,$ and $w$ are the vertical, shear angle, and longitudinal displacements. The kernel $g$ represents the memory effect acting only on the shear displacement. We note that when $R\rightarrow\infty$, then $\mathrm{l}\rightarrow0$ and the Bresse model reduces to well-known Timoshenko beam equations. \\
	
	Hago et al.  in \cite{hugo-Sare-racke} showed that the Timoshenko system with history type damping is not exponentially stable under Cattaneo's law, while under Fourier's law,  an exponential stability can be only attained once the speeds are equal. Moreover, no decay rate has been discussed if the speeds are different. This result has been recently improved by  Fatori et al in \cite{Lucifatori}, where an exponential stability is obtained with  Cattaneo's law if and only if a new condition on the wave speed of propagation is verified. Otherwise,  an optimal  energy decay rate of  type $\frac{1}{\sqrt{t}}$ is obtained.  Santos et al. in \cite{santossoufyane} extend the results of  \cite{Lucifatori} and \cite{hugo-Sare-racke} to the Bresse system with only one infinite memory type damping and under mixed boundary conditions. They  first proved that the system is exponentially stable if and only if  the three waves propagate with the same speed. Moreover, when at least two waves propagate with the same speed, a polynomial energy decay rate was established but this case has only a mathematical sense. Therefore, these results are very interesting but not complete. Indeed, from one hand, in the important physical natural case when the three waves have distinct speeds, no decay is discussed by Santos et al. On the other hand, the used techniques in all previously cited papers can not be adapted to prove the lack of exponential stability and to establish a polynomial decay rate in the interesting and difficult case of fully Dirichlet boundary conditions. 
	
	The purpose of this paper is to study the Bresse system in the presence of history type and/or heat conduction given by Cattaneo's law acting in the shear angle displacement equation and under fully Dirichlet or mixed boundary conditions.  We limit our attention to the case of fully Dirichlet boundary conditions since our study can be easily adapted to the other mixed boundary conditions. Besides the mathematical case when all the speeds are equal, we treat   the interesting physical case  when the three waves all propagate with different speeds. In fact, from the physical interpretation, we remark that the speeds of propagation of the three waves given by $\frac{\rho_1}{k_1}$, $\frac{\rho_2}{k_2}$, and $\frac{\rho_1}{k_3}$ are  all distinct. Our study is divided into two main parts.  First, we study the stability of the elastic Bresse system with only one memory type damping under fully Dirichlet  boundary  conditions.  When the speeds of the waves are all equal, we ensure the exponential decay found in \cite{santossoufyane}. 
	On the contrary,  using a spectrum method we  prove the lack of uniform stability.  
	Moreover,  when only two waves propagate with the same speed, using a frequency domain approach combining with a multiplier  method, we establish the energy decay rate of type $\frac{1}{t}$. Finally, in the interesting physical case, when the whole three waves propagate with different speeds, we prove an energy decay rate of type $\frac{1}{\sqrt{t}}$(see Theorem \ref{TH4.2} and Theorem \ref{TH4.3}). In these cases,  we conjecture the optimality of the energy decay rate.  Next, we adapt the study  of the  stability of the elastic Bresse system \eqref{eqq1.1} to the thermo-elastic Bresse system \eqref{eqq1.1'}.

	
	Last but not least, in addition to the previously cited papers,  we rapidly  recall some  previous studies done on the Bresse system. The stability of the elastic Bresse system with different kind of damping has been studied    in  \cite{LR4}, \cite{wehbey},  \cite{FatoriBresse}, \cite{AlabauBresse},   \cite{NounWehbe}, \cite{RiveraBresse}, \cite{aissaguesmia} and \cite{Wehbenadine} .  Guesmia et al.  in \cite{aissaguesmia}  considered Bresse system with infinite memories acting in the three equations of the system. They established  asymptotic stability results under some conditions on the relaxation functions regardless the speeds of propagation. 
	Furthermore, thermal stabilization of the Bresse system has been studied in \cite{RiveraBresse}, \cite{LR4}, \cite{Wehbenadine}. In \cite{LR4}, Liu and Rao  considered the Bresse system with two thermal dissipation laws.
	The results of \cite{LR4} are  improved by Fatori and Rivera in  \cite{RiveraBresse}
	where they studied the stability of Bresse system
	with one distributed temperature dissipation law operating on the angle displacement equation.
	Recently, Najdi and Wehbe in \cite{Wehbenadine}  extended and improved the results of \cite{RiveraBresse} when the thermal dissipation is locally distributed. \\

	This paper is organized as follows: In Section \ref{se2}, we prove the well-posedness   of system
	\eqref{eqq1.1} with fully Dirichlet boundary conditions. In Section \ref{strong}, we prove the strong stability of the system in the lack of the compactness of the resolvent of the generator.
	In Section \ref{se3}, we prove the exponential stability of the system on condition that the waves propagate
	with equal speeds. In Section \ref{se4}, we prove that the Bresse system considered with fully Dirichlet boundary conditions is non-uniformly stable when the speeds of the propagation of the waves are different. More precisely, we consider the reduced Timchenko system with fully Dirichlet boundary conditions and prove that an infinite number of  eigenvalues approach the imaginary axis. In Section \ref{se5}, if the waves propagate with different speeds,
	we prove the polynomial stability of the system. Indeed, if only two of the waves propagate with equal speeds,
	we prove a faster polynomial decay rate. Finally, in Section \ref{se6}, we adapt the results to the thermo-elastic Bresse
	system \eqref{eqq1.1'}.
	
	\section{Elastic Bresse system with one memory type control}\label{noheat}
	
	In this section, we study the stability of Bresse system with only one infinite memory damping acting in the equation about the shear angle displacement under fully Dirichlet boundary conditions (our study can be easily adapted to other mixed boundary conditions). The system is governed by the following partial differential equations:
	\begin{equation}\label{eqq1.1}
	\left\{
	\begin{array}{lll}
	\displaystyle{
		\rho_1\varphi_{tt}-k_1 \left(\varphi_x+\psi+\mathrm{l} w\right)_x-\mathrm{l}k_3\left(w_x-\mathrm{l}\varphi\right)=0,}\\ \\
	\displaystyle{\rho_2 \psi_{tt}-k_2\psi_{xx}+k_1\left(\varphi_x+\psi+\mathrm{l}w\right)+\int_0^{+\infty} g\left(s\right)\psi_{xx}\left(x,t-s\right)ds=0,}\\ \\
	\displaystyle{\rho_1w_{tt}-k_3\left(w_x-\mathrm{l}\varphi\right)_x+\mathrm{l}k_1\left(\varphi_x+\psi+\mathrm{l}w\right)=0,}
	\end{array}
	\right.
	\end{equation}
	with the following boundary conditions:
	\begin{equation}\label{DDD}
	\varphi\left(0,t\right)=\varphi\left(L,t\right)=\psi\left(0,t\right)=\psi\left(L,t\right)=
	w\left(0,t\right)=w\left(L,t\right)=0,
	\end{equation}
	and the following  initial conditions:
	$$
	\begin{array}{lll}
	\varphi\left(\cdot,0\right)=\varphi_0\left(\cdot\right),\ \psi\left(\cdot,-t\right)=\psi_0\left(\cdot,t\right),\
	w\left(\cdot,0\right)=w_0\left(\cdot\right),\\
	\varphi_t\left(\cdot,0\right)=\varphi_1\left(\cdot\right),\
	\psi_t\left(\cdot,0\right)=\psi_1\left(\cdot\right),\
	w_t\left(\cdot,0\right)=w_1\left(\cdot\right). \\
	\end{array}
	$$
	\begin{rk}
		The mixed boundary conditions make the calculations easier because they do not introduce point-wise terms when we apply
		the multiplicative techniques. However, in the case of fully Dirichlet boundary conditions, the calculations are more complicated because the
		boundary terms does not vanish.
	\end{rk}
	\subsection{Well-posedness of the problem}\label{se2}
	
	In this part, using a semi-group approach, we establish well-posedness result for the system  \eqref{eqq1.1}-\eqref{DDD} under condition (H) imposed
	into the relaxation function. For this purpose, similar to \cite{Constantine} and \cite{Mauro},  we introduce the new variable
	\begin{equation}\label{eta}
	\eta\left(x,t,s\right):=\psi\left(x,t\right)-\psi\left(x,t-s\right),\quad\text{in} \ \left(0,L\right)\times\mathbb{R}_{+}\times\mathbb{R}_{+}.
	\end{equation}
	Then, system  \eqref{eqq1.1}-\eqref{DDD} becomes
	\begin{equation}\label{eqq2.1}
	\left\{
	\begin{array}{lll}
	\displaystyle{
		\rho_1\varphi_{tt}-k_1 \left(\varphi_x+\psi+\mathrm{l} w\right)_x-\mathrm{l} k_3\left(w_x-\mathrm{l}\varphi\right)=0,}\\ \\
	\displaystyle{\rho_2 \psi_{tt}-\left(k_2-\int_0^{+\infty}g\left(s\right)ds\right)\psi_{xx}+k_1\left(\varphi_x+\psi+\mathrm{l}w\right)-\int_0^{+\infty} g\left(s\right)\eta_{xx}ds=0,}\\ \\
	\displaystyle{\rho_1w_{tt}-k_3\left(w_x-\mathrm{l}\varphi\right)_x+\mathrm{l}k_1\left(\varphi_x+\psi+\mathrm{l}w\right)=0,}
	\\ \\
	\displaystyle{\eta_t+\eta_s-\psi_t=0,}
	\end{array}
	\right.
	\end{equation}
	with the boundary conditions
	\begin{equation}\label{eqq2.3}
	\begin{array}{lll}
	\varphi\left(0,\cdot\right)=\varphi\left(L,\cdot\right)=\psi\left(0,\cdot\right)=\psi\left(L,\cdot\right)=w\left(0,\cdot\right)
	=w\left(L,\cdot\right)=0\quad&\text{in }\mathbb{R}_{+},\\
	\eta\left(0,\cdot,\cdot\right)=\eta\left(L,\cdot,\cdot\right)=0\quad&\text{in }\mathbb{R}_{+}\times\mathbb{R}_{+},\\
	\eta\left(\cdot,\cdot,0\right)=0&\text{in} \left(0,L\right)\times\mathbb{R}_{+},
	\end{array}
	\end{equation}
	and initial conditions
	\begin{equation}\label{eqq2.2}
	\begin{array}{lll}
	\varphi\left(\cdot,0\right)=\varphi_0\left(\cdot\right),\ \psi\left(\cdot,-t\right)=\psi_0\left(\cdot,t\right),\
	
	w\left(\cdot,0\right)=w_0\left(\cdot\right),
	& \\
	\varphi_t\left(\cdot,0\right)=\varphi_1\left(\cdot\right),\
	\psi_t\left(\cdot,0\right)=\psi_1\left(\cdot\right),\
	w_t\left(\cdot,0\right)=w_1\left(\cdot\right),  & \\
	\eta^0\left(\cdot,s\right):=\eta\left(\cdot,0,s\right)=\psi_0\left(\cdot,0\right)-\psi_0\left(\cdot,s\right)&\text{in }\left(0,L\right),\ s\geq 0.
	\end{array}
	\end{equation}
	The energy of system \eqref{eqq2.1}-\eqref{eqq2.3} is given by
	\begin{equation}\label{energy}
	\begin{array}{ll}
	\displaystyle{E\left(t\right)}&= \displaystyle{\frac{1}{2}\bigg\{\int_0^L\left(\rho_1\left|\varphi\right|^{2}+\rho_2\left|\psi_t\right|^2+\rho_1\left|w_t\right|^{2}
		+k_1\left|\varphi_x+\psi^2+\mathrm{l} w\right|^{2}+\widetilde{k}_2\left|\psi_x\right|^{2}\right)dx}\\ \\ & \displaystyle{+k_3\int_0^L \left|w_x-\mathrm{l} \varphi\right|^{2}dx+\int_0^L\int_0^{+\infty}g\left(s\right)\left|\eta_x\right|^{2}dx ds}\bigg\}.
	\end{array}
	\end{equation}
	Then a straightforward computation gives
	\begin{equation}\label{derivative of energy}
	E'\left(t\right)=\frac{1}{2}\int_0^L\int_0^{+\infty}g'\left(s\right)\left|\eta_x\right|^2ds dx\leq 0.
	\end{equation}
	Thus, the system  \eqref{eqq2.1}-\eqref{eqq2.3}  is dissipative in the sense that its energy is non increasing with respect to the time t. Now, we define  the  energy space $\mathcal{H}$ by
	\begin{equation*}
	\mathcal{H}=\left(H_0^1\left(0,L\right)\right)^3\times\left(L^2\left(0,L\right)\right)^3\times L^2_g\left(\mathbb{R}_+,H^1_0\right)
	\end{equation*}
	where $L^2_g\left(\mathbb{R}_+,H^1_0\right)$ denotes the Hilbert space endowed with the inner product
	\begin{equation*}
	\left(\eta^1,\eta^2\right)_{g}=\int_0^{L} \int_0^{+\infty} g\left(s\right)\eta_x^1\left(x,s\right)\eta_x^2\left(x,s\right)dsdx.
	\end{equation*}
	The energy space $\mathcal{H}$ is endowed with the following norm
	\begin{equation}\label{norm1}
	\begin{array}{ll}
	\displaystyle{\|U\|_{\mathcal{H}}^2}&= \displaystyle{\|(v^1,v^2,v^3,v^4,v^5,v^6,v^7)\|_{\mathcal{H}}^{2}}\\
	&=\rho_1\left\|v^4\right\|^{2}+\rho_2\left\|v^5\right\|^2+\rho_1\left\|v^6\right\|^{2}
	+k_1\left\|v^1_x+v^2+\mathrm{l} v^3\right\|^{2}+\widetilde{k}_2\left\|v^2_x\right\|^{2} \\
	&+k_3\left\|v^3_x-\mathrm{l} v^1\right\|^{2}+\left\|v^7\right\|^{2}_{g}
	\end{array}
	\end{equation}
	where  $\|\cdot\|$ and $\|\cdot\|_{g}$ denote  the norms of $L^2\left(0,L\right)$ and $L^2_g\left(\mathbb{R}_+,H^1_0\right)$ respectively.\\
	Next, we define the linear operator $\mathcal{A}$ in $\mathcal{H}$ by
	\begin{equation*}
	\begin{array}{l}
	D\left(\mathcal{A}\right)=\bigg\{\ U\in\mathcal{H} \ |\  v^1,v^3\in H^2\left(0,L\right),\  v^4,v^5,v^6\in  H^1_0\left(0,L\right),\ v^7_s\in L^2_g\left(\mathbb{R}_+,H^1_0\right), \\ \hspace{3.5cm} \ v^2+\int_0^{+\infty}g\left(s\right)v^7ds\in H^2\left(0,L\right)\cap H^1_0\left(0,L\right),\ v^7\left(x,0\right)=0\bigg\}
	\end{array}
	\end{equation*}
	and
	\begin{equation}\label{generator}
	\mathcal{A}\left(\begin{array}{l}
	v^1\\ v^2\\ v^3\\ v^4\\ v^5\\ v^6\\ v^7
	\end{array}\right)=\left(\begin{array}{c}
	v^4\\ v^5\\ v^6\\
	\rho_1^{-1}\left(k_1 \left(v^1_x+v^2+\mathrm{l} v^3\right)_x+\mathrm{l} k_3\left(v^3_x-\mathrm{l} v^1\right)\right)\\
	\rho_2^{-1} \left(\widetilde{k}_2 v^2_{xx}-k_1\left(v^1_x+v^2+\mathrm{l} v^3\right)+\int_0^{+\infty} g\left(s\right)v^7_{xx}ds\right)\\
	\rho_1^{-1}\left(k_3\left(v^3_x-\mathrm{l} v^1\right)_x-\mathrm{l} k_1\left(v^1_x+v^2+\mathrm{l} v^3\right)\right)\\
	v^5-v^7_s
	\end{array}\right)
	\end{equation}
	for all $U=\left(v^1,v^2,v^3,v^4,v^5,v^6,v^7\right)^{\mathsf{T}}\in D\left(\mathcal{A}\right)$. If $U=\left(\varphi,\psi,w,\varphi_t,\psi_t,w_t,\eta\right)^{\mathsf{T}}$ is the state of \eqref{eqq2.1}-\eqref{eqq2.3}, then the Bresse beam system is transformed into a first order evolution
	equation on the Hilbert space $\mathcal{H}$:
	\begin{equation}\label{Cauchy}
	\left\{
	\begin{array}{c}
	U_t=\mathcal{A}U,\\
	U\left(0\right)=U^0
	\end{array}
	\right.
	\end{equation}
	where
	\begin{equation*}
	U^0\left(x\right)=\left(\varphi_0\left(x\right),\psi_0\left(x,0\right),w_0\left(x\right),
	\varphi_1\left(x\right),\psi_1\left(x\right),w_1\left(x\right),\eta^0\left(x,\cdot\right)\right)^{\mathsf{T}}.
	\end{equation*}
	\begin{rk}
		It is easy to see that  there exists a positive constant $k_0'$ such that
		\begin{equation}\label{eqq2.6}
		k_1\left\|\varphi_x+\psi+\mathrm{l}w\right\|^2+\widetilde{k}_2\left\|\psi_x\right\|^2+k_3\left\|w_x-\mathrm{l}\varphi\right\|^2
		\leq k_0'\left(\left\|\varphi_x\right\|^2+\left\|\psi_x\right\|^2+\left\|w_x\right\|^2\right).
		\end{equation}
		On the other hand, under hypothesis  {\rm (H)}, as  $\widetilde{k}_2>0$, we can show  by a contradiction argument the existence of a positive constant $k_0$  such that, for any $\left(\varphi,\psi,w\right)\in \left(H_0^1\left(\left]0,L\right[\right)\right)^3$,
		\begin{equation}\label{eqq2.5}
		\displaystyle{k_0\left(\left\|\varphi_x\right\|^2+\left\|\psi_x\right\|^2+\left\|w_x\right\|^2\right)\leq
			k_1\left\|\varphi_x+\psi+\mathrm{l}w\right\|^2+\widetilde{k}_2\left\|\psi_x\right\|^2+k_3\left\|w_x-\mathrm{l}\varphi\right\|^2}.
		\end{equation}
		Therefore, the norm on the energy space $\mathcal{H}$ given in \eqref{norm1} is equivalent to the usual norm on $\mathcal{H}$.
	\end{rk}
	\begin{prop}\label{TH2.1}
		Under hypothesis  {\rm (H)}, the operator $\mathcal{A}$ is m-dissipative in the energy space $\mathcal{H}$.
	\end{prop}
	\begin{proof}
		For all $U\in D\left(\mathcal{A}\right)$, by a straight forward calculation, we have
		\begin{equation}\label{eqq2.8}
		\Re\left(\left<\mathcal{A}U,U\right>_{\mathcal{H}}\right)=\frac{1}{2}
		\int_0^{L}\int_0^{+\infty}g'\left(s\right)\left| v^7_x\right|^2dsdx.
		\end{equation}
		As $g$ is non-increasing we get that  $\mathcal{A}$ is  dissipative. Now let  $F=\left(f^1,f^2,f^3,f^4,f^5,f^6,f^7\right)^\mathsf{T}\in\mathcal{H}$, we prove the existence of $$U=\left(v^1,v^2,v^3,v^4,v^5,v^6,v^7\right)^\mathsf{T}\in D\left(\mathcal{A}\right)$$  unique solution of the equation
		$$-\mathcal{A}U=F.$$
		Equivalently, we have the following system
		\begin{eqnarray}
		-v^4&=&f^1, \quad\label{eqq2.9}
		\\
		-v^5&=&f^2,\quad\label{eqq2.10}
		\\
		-v^6&=&f^3, \quad\label{eqq2.11}
		\\
		-k_1 \left[v^1_x+v^2+\mathrm{l} v^3\right]_x-\mathrm{l} k_3\left[v^3_x-\mathrm{l} v^1\right]
		&=&\rho_1 f^4,\quad\label{eqq2.12}
		\\
		-\widetilde{k}_2 v^2_{xx}+k_1\left[v^1_x+v^2+\mathrm{l} v^3\right]-\int_0^{+\infty} g\left(s\right)v^7_{xx}ds &=&\rho_2 f^5,\label{eqq2.13}
		\\
		-k_3\left[v^3_x-\mathrm{l} v^1\right]_x+\mathrm{l} k_1\left[v^1_x+v^2+\mathrm{l} v^3\right]&=&\rho_1 f^6 ,\quad\label{eqq2.14}
		\\
		v^7_{s}-v^5 &=& f^7.\quad\label{eqq2.15}
		\end{eqnarray}
		From \eqref{eqq2.15} and \eqref{eqq2.10}, we can
		determine
		%
		%
		%
		%
		%
		%
		%
		%
		\begin{equation}\label{eqq2.16}
		v^7\left(x,s\right)=-sf^2\left(x\right)+ \int_0^sf^7\left(x,\tau\right)d\tau.
		\end{equation}
		It is clear that $v^7\left(x,0\right)=0$ and $v^7_s\in  L^2_g\left(\mathbb{R}_+,H^1_0\right)$. To prove that $v^7\in  L^2_g\left(\mathbb{R}_+,H^1_0\right)$ let $T,\epsilon>0$ be arbitrary. Using  hypothesis  {\rm (H)}, we have 
		\begin{equation}\label{m10}
		\begin{array}{ll}
		\displaystyle{\int_{\epsilon}^{T} g\left(s\right)\left\|v^7_x\right\|^2 ds} \leq \displaystyle{-\frac{1}{c}\int_{\epsilon}^{T}g'\left(s\right)\left\|v^7_x\right\|^2 ds}\\ \\ \displaystyle{\leq -\frac{g\left(T\right)}{c} \left\|v^7_x\left(\cdot,T\right)\right\|^2 +\frac{g\left(\epsilon\right)}{c} \left\|v^7_x\left(\cdot,\epsilon\right)\right\|^2+ \frac{2}{c}\int_{\epsilon}^{T}g\left(s\right)\Re\left\{\intspace v^7_x\overline{v^7_{xs}}dx\right\}ds}.
		\end{array}
		\end{equation}
		Using the hypothesis on $g$ and the fact that $v^7\left(x,0\right)=0$, then
		as $T\to+\infty$ and $\epsilon\to 0,$ we obtain from \eqref{m10}
		\begin{equation*}
		\begin{array}{ll}
		\displaystyle{\int_{0}^{+\infty} g\left(s\right)\left\|v^7_x\right\|^2 ds}& \displaystyle{\leq \frac{2}{c}\int_{0}^{+\infty}g\left(s\right)\Re\left\{\intspace v^7_x\overline{v^7_{xs}}dx\right\}ds}\\ \\
		&\displaystyle{\leq\frac{1}{2}\int_{0}^{+\infty} g\left(s\right)\left\|v^7_x\right\|^2 ds+\frac{2}{c^2}\int_{0}^{+\infty} g\left(s\right)\left\|v^7_{sx}\right\|^2 ds}. 
		\end{array}
		\end{equation*}
		It follows that 
		$$\int_{0}^{+\infty} g\left(s\right)\left\|v^7_x\right\|^2 ds\leq\frac{4}{c^2}\int_{0}^{+\infty} g\left(s\right)\left\|v^7_{sx}\right\|^2 ds<+\infty.$$
		Therefore $v^7\in  L^2_g\left(\mathbb{R}_+,H^1_0\right)$.
		Now, inserting \eqref{eqq2.16} in \eqref{eqq2.12}-\eqref{eqq2.14}, we get
		\begin{equation}\label{eqq2.18}
		\left\{
		\begin{array}{ll}
		\displaystyle{-k_1 \left[v^1_x+v^2+\mathrm{l} v^3\right]_x-\mathrm{l} k_3\left[v^3_x-\mathrm{l} v^1\right]
			=\rho_1f^4,}
		\\ \\
		\displaystyle{-\underline{k}_2 v^2_{xx}+k_1\left[v^1_x+v^2+\mathrm{l} v^3\right]=\rho_2f^5+ \int_0^{+\infty}g\left(s\right)\left(-sf^2\left(x\right)+ \int_0^sf^7\left(x,\tau\right)d\tau\right)_{xx} ds},
		\\ \\
		\displaystyle{-k_3\left[v^3_x-\mathrm{l} v^1\right]_x+\mathrm{l} k_1\left[v^1_x+v^2+\mathrm{l}v^3\right]=\rho_1f^6 ,}
		\end{array}
		\right.
		\end{equation}
		where $ \underline{k}_2=k_2-\int_0^{+\infty}e^{-s}g\left(s\right)ds>0$. \\
		Using Lax-Milgram Theorem (see \cite{pazy}), we deduce that \eqref{eqq2.18} admits a unique solution in $\left( H^1_0\left(0,L\right)\right)^3$. Thus, using \eqref{eqq2.16} and classical regularity arguments, we conclude that $-\mathcal{A}U=F$ admits a unique solution $U\in D\left(\mathcal{A}\right)$ and $0\in \rho(\mathcal{A})$. Since $D(\mathcal{A})$ is dense in $\mathcal{H}$ then, by the resolvent identity, for small $\lambda > 0$, we have $R(\lambda I -\mathcal{A} ) = \mathcal{H}$ (see Theorem 1.2.4  in \cite{liu:99}) and $\mathcal{A}$ is m-dissipative in $\mathcal{H}$. The proof is thus complete.
	\end{proof}
	
	Thanks to  Lumer-Phillips Theorem (see \cite{liu:99,pazy}), we deduce that $\mathcal{A}$ generates a  $C_0$-semigroup of contraction $e^{t\mathcal{A}}$ in $\mathcal{H}$ and therefore  problem \eqref{eqq2.1}-\eqref{eqq2.3} is well-posed. Then we have the following result:
	\begin{thm}\label{TH2.2}
		Under hypothesis  {\rm (H)}, for any $U^0\in\mathcal{H}$, problem \eqref{Cauchy}  admits a unique weak solution
		$$U\in C\left(\mathbb{R}_{+};\mathcal{H}\right).$$
		Moreover, if $U^0\in D\left(\mathcal{A}\right),$ then
		$$U\in C\left(\mathbb{R}_{+};D\left(\mathcal{A}\right)\right) \cap C^1\left(\mathbb{R}_{+};\mathcal{H}\right).$$
	\end{thm}
	\subsection{Strong stability of the system}\label{strong}
	
	In this part, we use a general criteria of Arendt-Batty in \cite{arendt:88} to show the strong stability of the $C_0$-semigroup
	$e^{t\mathcal{A}}$ associated to the Bresse system \eqref{eqq2.1}-\eqref{eqq2.3} in the absence of the compactness of the resolvent of $\mathcal{A}$. 
	Our main result is the following theorem:
	\begin{thm}\label{strongstability}
		Assume that {\rm (H)} is true. Then, the $C_0$-semigroup $e^{t\mathcal{A}}$ is strongly stable in $\mathcal{H}$; i.e, for all $U^0\in\mathcal{H}$, the solution of \eqref{Cauchy} satisfies
		$$
		\lim_{t\to+\infty}\left\|e^{t\mathcal{A}} U^0\right\|_{\mathcal{H}}=0.
		$$
	\end{thm}
	For the proof of Theorem  \ref{strongstability}, we need the following two lemmas.
	\begin{lem}\label{Th3.1}
		Under hypothesis {\rm (H)}, we have
		\begin{equation}\label{eqq3.4}
		\ker\left(\mathit{i}\lambda-\mathcal{A}\right)=\{0\}\,\,\,\,  {\rm for\, all } \,\, \lambda\in\mathbb{R}.
		\end{equation}
	\end{lem}
	\begin{proof} From  Proposition \ref{TH2.1}, we deduce that $0\in\rho\left(\mathcal{A}\right)$.  We still need to show the result for $\lambda\in\mathbb{R^*}$.
		Suppose that  there exists a real number $\lambda\neq 0$  and
		$ U=\left(v^1,v^2,v^3,v^4,v^5,v^6,v^7\right)^{\mathsf{T}}\in D\left(\mathcal{A}\right)$ such that
		\begin{equation}\label{eqq3.5}
		\mathcal{A}U=\mathit{i}\lambda U.
		\end{equation}
		Then, we have
		\begin{equation*}
		\Re\left(\left<\mathcal{A}U,U\right>_{\mathcal{H}}\right)=\frac{1}{2}
		\int_0^{L}\int_0^{+\infty}g'\left(s\right)\left| v^7_x\right|^2dsdx=0.
		\end{equation*}
		Due to hypothesis {\rm (H)},  it follows that
		\begin{equation*}
		\int_0^{L}\int_0^{+\infty}g\left(s\right)\left| v^7_x\right|^2dsdx=0.
		\end{equation*}
		This implies that 
		%
		%
		%
		\begin{equation*}
		v^7=0.
		\end{equation*}
		Now equation \eqref{eqq3.5} is equivalent to
		\begin{eqnarray}
		v^4&=&\mathit{i}\lambda v^1, \label{eqq3.6}
		\\
		v^5&=&\mathit{i}\lambda v^2,\label{eqq3.7}
		\\
		v^6&=&\mathit{i}\lambda v^3,\label{eqq3.8}
		\\
		k_1 \left[v^1_x+v^2+\mathrm{l} v^3\right]_x+\mathrm{l} k_3\left[v^3_x-\mathrm{l} v^1\right]
		&=&\mathit{i}\rho_1\lambda v^4,\label{eqq3.9}
		\\
		\widetilde{k}_2 v^2_{xx}-k_1\left[v^1_x+v^2+\mathrm{l} v^3\right]&=&\mathit{i}\rho_2\lambda v^5,\label{eqq3.10}
		\\
		k_3\left[v^3_x-\mathrm{l}v^1\right]_x-\mathrm{l}k_1\left[v^1_x+v^2+\mathrm{l} v^3\right]&=&\mathit{i}\rho_1\lambda v^6,\label{eqq3.11}
		\\
		v^5&=&0.\label{eqq3.12}
		\end{eqnarray}
		Using equations \eqref{eqq3.12}, \eqref{eqq3.7}, \eqref{eqq3.10} and the fact that $v^2,\ v^5\in H^1_0\left(0,L\right)$, we get
		\begin{equation}\label{eqq3.13}
		v^2=v^5=0 \text{ and } v^1_x+\mathrm{l} v^3=0.
		\end{equation}
		Inserting \eqref{eqq3.6}, \eqref{eqq3.8} and \eqref{eqq3.13} into equations \eqref{eqq3.9} and \eqref{eqq3.11}, we get
		\begin{eqnarray*}
			-k_3 v^1_{xx}+\left(\rho_1\lambda^2-\mathrm{l}^2k_3\right) v^1&=&0,\\
			k_3v^3_{xx}+\left(\rho_1\lambda^2+\mathrm{l}^2k_3\right) v^3&=&0,\\
			v^1(0)=v^1(L)=v^3(0)=v^3(L)&=&0.
		\end{eqnarray*}
		By direct calculation, we deduce that $v^1=v^3=0$  and therefore $U=0$. The proof is thus complete.
	\end{proof}
	\begin{lem}\label{compactness}
		Under hypothesis  {\rm (H)}, $\mathit{i}\lambda  -\mathcal{A}$ is surjective for all $\lambda\in\mathbb{R}$.
	\end{lem}
	\begin{proof}
		Since $0\in\rho\left(\mathcal{A}\right)$.  We still need to show the result for $\lambda\in\mathbb{R^*}$. For any  $$F=\left(f^1,f^2,f^3,f^4,f^5,f^6,f^7\right)^{\mathsf{T}}\in\mathcal{H},\ \lambda\in \mathbb{R}^*,$$ we prove the existence of $$U=\left(v^1,v^2,v^3,v^4,v^5,v^6,v^7\right)^{\mathsf{T}}\in D\left(\mathcal{A}\right)$$ solution of the following equation
		$$\left(\mathit{i}\lambda -\mathcal{A}\right)U=F.$$
		Equivalently, we have the following system
		\begin{eqnarray}
		\mathit{i}\lambda v^1-v^4&=&f^1, \label{eqq3.14}
		\\
		\mathit{i}\lambda v^2-v^5&=&f^2,\label{eqq3.15}
		\\
		\mathit{i}\lambda v^3-v^6&=&f^3,\label{eqq3.16}
		\\
		\rho_1\mathit{i}\lambda v^4-k_1 \left[v^1_x+v^2+\mathrm{l} v^3\right]_x-\mathrm{l}k_3\left[v^3_x-\mathrm{l} v^1\right]
		&=&\rho_1 f^4,\label{eqq3.17}
		\\
		\rho_2\mathit{i}\lambda v^5 -\widetilde{k}_2 v^2_{xx}+k_1\left[v^1_x+v^2+\mathrm{l} v^3\right]-\int_0^{+\infty} g\left(s\right)v^7_{xx} ds&=&\rho_2 f^5,\label{eqq3.18}
		\\
		\rho_1\mathit{i}\lambda v^6-k_3\left[v^3_x-\mathrm{l} v^1\right]_x+\mathrm{l} k_1\left[v^1_x+v^2+\mathrm{l} v^3\right]&=&\rho_1 f^6 ,\label{eqq3.19}
		\\
		\mathit{i}\lambda v^7 +v^7_{s}-v^5 &=& f^7.\quad\label{eqq3.20}
		\end{eqnarray}
		From \eqref{eqq3.20} and \eqref{eqq3.15}, we have
		$$v^7_{s}+\mathit{i}\lambda v^7 =\mathit{i}\lambda  v^2-f^2+f^7.$$
		It follows that
		\begin{equation}\label{eqq3.21}
		v^7\left(x,s\right)=\left(1-e^{-\mathit{i}\lambda s}\right)v^2\left(x\right)+\frac{\mathit{i}}{\lambda}\left(1-e^{-\mathit{i}\lambda s}\right)f^2\left(x\right)+\int_0^s e^{\mathit{i}\lambda\left( \tau-s\right)}f^7\left(x,\tau\right)d\tau.
		\end{equation}
		From \eqref{eqq3.14}-\eqref{eqq3.16}, we have
		\begin{equation}\label{eqq3.22}
		v^4=\mathit{i}\lambda v^1-f^1,\quad v^5=\mathit{i}\lambda v^2-f^2,\quad v^6=\mathit{i}\lambda v^3-f^3.
		\end{equation}
		Inserting \eqref{eqq3.21} and \eqref{eqq3.22} in \eqref{eqq3.17}-\eqref{eqq3.19}, we get
		\begin{equation}\label{eqq3.23}
		\left\{
		\begin{array}{ll}
		\displaystyle{-\lambda^2 v^1-k_1 \rho_1^{-1}\left[v^1_x+v^2+\mathrm{l} v^3\right]_x-\mathrm{l} k_3\rho_1^{-1}\left[v^3_x-\mathrm{l}v^1\right]
			=h^1,}
		\\ \\
		\displaystyle{-\lambda^2 v^2-\breve{k}_2 \rho_2^{-1}v^2_{xx}+k_1\rho_2^{-1}\left[v^1_x+v^2+\mathrm{l} v^3\right]=h^2 ,}
		\\ \\
		\displaystyle{-\lambda^2 v^3-k_3\rho_1^{-1}\left[v^3_x-\mathrm{l} v^1\right]_x+\mathrm{l} k_1\rho_1^{-1}\left[v^1_x+v^2+\mathrm{l} v^3\right]=h^3 ,}
		\end{array}
		\right.
		\end{equation}
		where
		$$\breve{k}_2=k_2-\int_0^{+\infty} e^{-\mathit{i}\lambda s}g\left(s\right)ds,$$
		and
		\begin{equation*}
		\left\{
		\begin{array}{ll}
		\displaystyle{h^1=f^4+\mathit{i}\lambda f^1,\quad h^3=f^6+\mathit{i}\lambda f^3,}\\ \\
		\displaystyle{h^2=f^5+ \mathit{i}\lambda f^2+\frac{\mathit{i}}{\lambda}\rho_2^{-1}\int_0^{+\infty}\left(1-e^{-\mathit{i}\lambda s}\right)g\left(s\right)ds f^2_{xx}+\rho_2^{-1}\int_0^{+\infty}g\left(s\right)\int_0^s e^{\mathit{i}\lambda\left(\tau-s\right)}f^7_{xx}d\tau ds}.
		\end{array}
		\right.
		\end{equation*}
		Define the operators
		\begin{equation*}
		\mathcal{L}v=\left(\begin{array}{c}
		\displaystyle{-k_1 \rho_1^{-1}\left(v^1_x+v^2+\mathrm{l} v^3\right)_x-\mathrm{l} k_3 \rho_1^{-1}\left(v^3_x-\mathrm{l}v^1\right)
		}
		\\ \\
		\displaystyle{-\breve{k}_2 \rho_2^{-1}v^2_{xx}+k_1\rho_2^{-1}\left(v^1_x+v^2+\mathrm{l} v^3\right) }
		\\ \\
		\displaystyle{-k_3\rho_1^{-1}\left(v^3_x-\mathrm{l} v^1\right)_x+\mathrm{l} k_1\rho_1^{-1}\left(v^1_x+v^2+\mathrm{l} v^3\right)}
		\end{array}\right),  \,\,\, \forall\ v=\left(v^1,v^2,v^3\right)^{\mathsf{T}}\in \bigg(H^1_0(0,L)\bigg)^3.
		\end{equation*}
		Using Lax-Milgram theorem, it is easy to show that $\mathcal{L}$ is an isomorphism from $(H^1_0(0,L))^3$ onto $(H^{-1}\left(0,L\right))^3$. Let $v=\left(v^1,v^2,v^3\right)^{\mathsf{T}}$ and  $h=\left(h^1,h^2,h^3\right)^{\mathsf{T}}$, then we transform system \eqref{eqq3.23} into the following form
		\begin{equation}\label{eqq3.28}
		v-\lambda^2\mathcal{L}^{-1}v=\mathcal{L}^{-1}h.
		\end{equation}
		Using the compactness embeddings from $L^2(0,L)$ into $H^{-1}(0,L)$ and from $H^1_0(0,L)$ into $L^{2}(0,L)$ we deduce that the operator $\mathcal{L}^{-1}$
		is compact from $L^2(0,L)$ into $L^2(0,L)$. Consequently, by Fredholm alternative, proving the existence of $v$ solution of \eqref{eqq3.28} reduces to proving the injectivity of the operator $Id-\lambda^2\mathcal{L}^{-1}$. Indeed, if $\tilde{v}=\left(\tilde{v}^1,\tilde{v}^2,\tilde{v}^3\right)^{\mathsf{T}}\in{\rm Ker}(Id-\lambda^2\mathcal{L}^{-1})$, then we have $\lambda^2 \tilde{v}-  \mathcal{L} \tilde{v}=0$. It follows that
		\begin{equation}\label{eqq3.29}
		\left\{
		\begin{array}{ll}
		\displaystyle{-\rho_1\lambda^2 \tilde{v}^1-k_1 \left[\tilde{v}^1_x+\tilde{v}^2+\mathrm{l} \tilde{v}^3\right]_x-\mathrm{l} k_3\left[\tilde{v}^3_x-\mathrm{l} \tilde{v}^1\right]
			=0,}
		\\ \\
		\displaystyle{-\rho_2\lambda^2 \tilde{v}^2-\breve{k}_2 \tilde{v}^2_{xx}+k_1\left[\tilde{v}^1_x+\tilde{v}^2+\mathrm{l} \tilde{v}^3\right]=0,}
		\\ \\
		\displaystyle{-\rho_1\lambda^2 \tilde{v}^3-k_3\left[\tilde{v}^3_x-\mathrm{l} \tilde{v}^1\right]_x+I k_1\left[\tilde{v}^1_x+\tilde{v}^2+\mathrm{l} \tilde{v}^3\right]=0.}
		\end{array}
		\right.
		\end{equation}
		Now, it is easy to see that if $(\tilde{v}^1,\tilde{v}^2,\tilde{v}^3)$ is a solution of system \eqref{eqq3.29} then the vector $\tilde{V}$ defined by  $\tilde{V}= \left(\tilde{v}^1,\tilde{v}^2,\tilde{v}^3,\mathit{i}\lambda \tilde{v}^1,\mathit{i}\lambda \tilde{v}^2,\mathit{i}\lambda \tilde{v}^3,\left(1-e^{-\mathit{i}\lambda s}\right) \tilde{v}^2\right)^{\mathsf{T}}$ belongs to $D(\mathcal{A})$ and we have $i\lambda \tilde{V}-\mathcal{A}\tilde{V}=0.$
		Therefore, by Lemma \ref{Th3.1}, we get $\tilde{V}=0$ and so ${\rm Ker}(Id-\lambda^2\mathcal{L}^{-1})=\{0\}$. Thanks to Fredholm alternative, the equation  \eqref{eqq3.28} admits a unique solution $v=(v^1,v^2,v^3) \in \left(H^1_0\left(0,L\right)\right)^3$. Thus, using \eqref{eqq3.21}, \eqref{eqq3.22} and a classical regularity arguments, we conclude that $\left(\mathit{i}\lambda -\mathcal{A}\right)U=F$ admits a unique solution $U\in D\left(\mathcal{A}\right)$. The proof is thus complete.
	\end{proof}
	
	\medskip
	
	\noindent{\bf Proof of Theorem \ref{strongstability}.} Following a general criteria of Arendt-Batty in \cite{arendt:88}, the $C_0-$semigroup $e^{t\mathcal{A}}$ of contractions is strongly  stable if $\mathcal{A}$ has no pure imaginary eigenvalues and $\sigma(\mathcal{A})\cap i\mathbb{R}$ is countable. By Lemma \ref{Th3.1}, the operator $\mathcal{A}$ has no pure imaginary eigenvalues and by Lemma \ref{compactness}, ${\rm R}(\mathit{i}\lambda  -\mathcal{A})=\mathcal{H}$  for all $\lambda\in\mathbb{R}$. Therefore, the closed graph theorem of Banach  implies that $\sigma(\mathcal{A})\cap i\mathbb{R}=\emptyset$. The proof is thus complete.
	\begin{rk}
		Using a multiplier method,  Santos et {\it al.} in \cite{santossoufyane} established the strong stability of Bresse system with only one infinite memory damping. Their study is only available  for one dimensional case. In Theorem \ref{strongstability}, our approach can be generalized to multi-dimensional spaces. In addition, our condition (H) on the relaxation function $g$ is weaker than that used in \cite{santossoufyane}.  
	\end{rk}
	\begin{rk}
		We mention \cite{RaoWehbe:Kirchoff} for a direct approach of the strong stability of
		Kirchhoff plates in the absence of compactness of the resolvent.
	\end{rk}
	\subsection{Exponential stability in the case $\frac{k_1}{\rho_1}=\frac{k_2}{\rho_2}$ and $k_1=k_3$}\label{se3}
	In this part, under hypothesis {\rm (H)}, we prove the exponential stability of the Bresse system \eqref{eqq1.1}-\eqref{DDD} provided that
	\begin{equation}\label{eqq3.1}
	\frac{k_1}{\rho_1}=\frac{k_2}{\rho_2}\,\,\, \text{ and }\,\,\, k_1=k_3.
	\end{equation}
	Our main result in this part is the following stability estimate:
	\begin{thm}\label{TH3.5}
		Assume that   \eqref{eqq3.1} is satisfied. Under hypothesis {\rm (H)},  the $C_0$-semigroup $e^{t\mathcal{A}}$ is exponentially stable, i.e, there exist constants $M\geq1,$ and $\epsilon>0$ independent of $U^0$ such that
		\begin{equation*}
		\left\|e^{t\mathcal{A}}U^0\right\|_{\mathcal{H}}\leq M e^{-\epsilon t}\left\|U^0\right\|_{\mathcal{H}}, \quad t\geq0.
		\end{equation*}
	\end{thm}
	According to \cite{Huang:85} and \cite{Pruss}, we have to check if the following conditions hold,
	\begin{equation*}
	\mathit{i}\mathbb{R}\subseteq\rho\left(\mathcal{A}\right) \,\,\,\,\,\,\,\,  {\rm (H1)},
	\end{equation*}
	and
	\begin{equation*}
	\sup_{\lambda\in\mathbb{R}}\left\|\left(i\lambda Id-\mathcal{A}\right)^{-1}\right\|_{\mathcal{L}\left(\mathcal{H}\right)}=O\left(1\right)
	\,\,\,\,\,\,\,\,  {\rm (H2)}.
	\end{equation*}
	Condition  $\mathit{i}\mathbb{R}\subseteq\rho\left(\mathcal{A}\right)$ is  already proved in Lemma \ref{Th3.1} and Lemma \ref{compactness}. We will prove condition (H2) by a contradiction argument. Suppose that there exist a sequence of real numbers $\left(\lambda_n\right)_n$, with  $|\lambda_n|\to+\infty,$ and a sequence of vectors
	\begin{equation}\label{eqq3.32}
	U_n=\left(v^1_n,v^2_n,v_n^3,v^4_n,v^5_n,v^6_n,v^7_n\right)^{\mathsf{T}}\in D\left(\mathcal{A}\right) \ \text{ with }\ \|U_n\|_{\mathcal{H}}=1
	\end{equation}
	such that
	\begin{equation}\label{eqq3.33}
	\mathit{i}\lambda_n U_n-\mathcal{A}U_n=\left(f^1_n,f^2_n,f_n^3,f^4_n,f^5_n,f^6_n,f^7_n\right)^{\mathsf{T}}\to 0\ \text{ in } \mathcal{H}.
	\end{equation}
	That we detail as
	\begin{eqnarray}
	\mathit{i}\lambda_n v^1_n-v^4_n&=&h^1_n,\label{eqq3.34}
	\\
	\mathit{i}\lambda_n v^2_n-v^5_n&=&h^2_n,\label{eqq3.35}
	\\
	\mathit{i}\lambda_n v^3_n-v^6_n&=&h^3_n,\label{eqq3.36}
	\\
	\rho_1\lambda_n^2 v^1_n+k_1 \left[\left(v^1_n\right)_x+v^2_n+\mathrm{l} v^3_n\right]_x+\mathrm{l} k_3\left[\left(v^3_n\right)_x-\mathrm{l} v^1_n\right]
	&=& h^4_n,\label{eqq3.37}
	\\
	\rho_2\lambda_n^2 v^2_n    +\widetilde{k}_2 \left(v^2_n\right)_{xx}-k_1\left[\left(v^1_n\right)_x+v^2_n+\mathrm{l} v^3_n\right]+\int_0^{+\infty} g\left(s\right)\left(v^7_n\right)_{xx} ds&=& h^5_n,\label{eqq3.38}
	\\
	\rho_1\lambda_n^2 v^3_n+k_3\left[\left(v^3_n\right)_x-\mathrm{l} v^1_n\right]_x-\mathrm{l} k_1\left[\left(v^1_n\right)_x+v^2_n+\mathrm{l} v^3_n\right]&=& h^6_n,\label{eqq3.39}
	\\
	\mathit{i}\lambda_n v^7_n +\left(v^7_n\right)_{s}-\mathit{i}\lambda_n v^2_n &=&h^7_n,\label{eqq3.40}
	\end{eqnarray}
	where
	\begin{equation*}
	\left\{
	\begin{array}{ll}
	h^1_n=f^1_n,\ h^2_n=f^2_n,\ h^3_n=f^3_n,\ h^7_n=f^7_n-f^2_n, \\ \\
	h^4_n=-\rho_1\left( f^4_n+\mathit{i}\lambda_nf^1_n\right),\ h^5_n=-\rho_2\left(  f^5_n+\mathit{i}\lambda_nf^2_n\right),\
	h^6_n=-\rho_1 \left( f^6_n+\mathit{i}\lambda_nf^3_n\right). 
	\end{array}
	\right.
	\end{equation*}
	In the following we will check the condition (H2) by finding a contradiction
	with \eqref{eqq3.32} such as $\left\| U_n\right\|_\mathcal{H} =o(1)$. For clarity, we divide the proof into several lemmas.
	From  now on, for simplicity, we drop the index $n$.

	\begin{lem}\label{LE3.1} Assume that hypothesis {\rm (H)} is verified. Then  we have
		\begin{equation}\label{eqq3.45}
		\intspace\inttime g\left(s\right)\left|v^7_x\right|^2ds dx=o\left(1\right).
		\end{equation}
	\end{lem}
	\begin{proof}
		Taking the inner product of \eqref{eqq3.33} with $U$ in $\mathcal{H}$. Then, using \eqref{eqq2.8} and the fact that $U$ is uniformly bounded in $\mathcal{H}$, we get
		\begin{equation}\label{eqq3.46}
		\frac{1}{2}\int_0^{L}\int_0^{+\infty}g'\left(s\right)\left| v^7_x\right|^2dsdx=\Re\left(\mathcal{A}U,U\right)_{\mathcal{H}}=-\Re\left(\mathit{i}\lambda U-\mathcal{A}U,U\right)_{\mathcal{H}}=o\left(1\right).
		\end{equation}
		Using condition (H)  into \eqref{eqq3.46}, 
		we obtain the desired asymptotic equation \eqref{eqq3.45}. Thus the proof is complete.
	\end{proof}
	\begin{lem}\label{LE3.2} Assume that hypothesis {\rm (H)} is verified. Then we have
		\begin{equation}\label{eqq3.47}
		\intspace\left|v^2_x\right|dx={o\left(1\right)}\,\,\,\, {\rm and }\,\,\,\, \intspace\left|\lambda  v^2\right|^2dx=o\left(1\right).
		\end{equation}
	\end{lem}
	\begin{proof}
		Multiplying \eqref{eqq3.40} by $\overline{v^2}$ in $L^2_g\left(\mathbb{R}_+,H_0^1\right)$. Then, using the fact that  $\left\| v^2 \right\|_g^2=g_0\left\| v^2_x \right\|^2$, $v^2_x$ is uniformly bounded in $L^2(0,L)$, $f^2$ converges to zero in $H^1_0(0,L)$ and $f^7$ converges to zero in $L^2_g\left(\mathbb{R}_+,H_0^1\right)$, we get
		\begin{equation}\label{eqq3.49}
		g^0\intspace\left|v^2_x\right|^2dx=\intspace\inttime g\left(s\right)v^7_x\overline{v_x^2}ds dx+\frac{1}{\mathit{i}\lambda}\intspace\inttime g\left(s\right) v^7_{xs}\overline{v^2_x} ds dx+
		\frac{o(1)}{\lambda}.
		\end{equation}
		Using by parts integration, condition {\rm (H)} and  the fact that $v^7\left(x,0\right)=0$, we get
		\begin{equation*}
		\frac{1}{\mathit{i}\lambda}\intspace\inttime g\left(s\right) v^7_{xs}\overline{v^2_x} ds dx=-\frac{1}{\mathit{i}\lambda}\int_0^{L} \int_0^{+\infty}
		g'\left(s\right)v^7_{x}\overline{{v}^2_x}dsdx.
		\end{equation*}
		Then, applying Holder's inequality in $L^2(0,L)$ and $L^2(0,+\infty)$ and using \eqref{eqq3.46} and the fact that $v^2_x$ is uniformly bounded in $L^2(0,L)$ and  $\displaystyle{\lim_{s\to0}\sqrt{g\left(s\right)}}$ exists, it follows that
		\begin{equation}\label{eqq3.48}
		\left|\frac{1}{\lambda}\intspace\inttime g\left(s\right) v^7_{xs}\overline{v^2_x} ds dx\right|\leq
		\frac{\displaystyle{\lim_{s\to0}\sqrt{g\left(s\right)}}}{\left|\lambda\right|}
		\left(\int_0^{L}\int_0^{+\infty}-g'\left(s\right)\left|v^7_x\right|^2dsdx\right)^{1/2} \left\|v^2_x\right\|=\frac{o(1)}{\lambda}.
		\end{equation}
		Using  Lemma \ref{LE3.1} and the fact that  $v^2_x$ is uniformly bounded in $L^2(0,L)$, we get 
		\begin{equation}\label{w1c}
		\left|\intspace\inttime g\left(s\right)v^7_x\overline{v_x^2}ds dx\right|=o\left(1\right).
		\end{equation}
		Using equation \eqref{eqq3.48} and \eqref{w1c} in equation \eqref{eqq3.49}, we get the first asymptotic estimate of \eqref{eqq3.47}. Now, multiplying \eqref{eqq3.38} by $\overline{v^2}$ in $L^2(0,L)$. Then, using the fact that  $v^2$ is uniformly bounded in $L^2(0,L)$, $f^2$ converges to zero in $H^1_0(0,L)$ and $f^5$ converges to zero in $L^2\left(0,L\right)$, we get 
		\begin{equation}\label{eqq3.51}
		\begin{array}{ll}
		\rho_2\lambda^2\intspace \left|v^2\right|^2dx=\widetilde{k}_2\intspace\left|v^2_x\right|^2dx+ \intspace\inttime g\left(s\right) v^7\overline{v^2}ds dx\\ \\ \hspace{2.6cm}+k_1\intspace \left( v^1_x+v^2+\mathrm{l} v^3\right)\overline{v^2}dx+o\left(1\right).
		\end{array}
		\end{equation}
		Using Lemma \ref{LE3.1}, the first estimation of \eqref{eqq3.47} and the fact that $ v^1_x$ is uniformly bounded in $L^2(0,L)$,
		$\left\| v^2 \right\|=o(1)$,   $v^2, v^3$  converge to zero in $L^2\left(0,L\right)$ in equation \eqref{eqq3.51} we obtain the second asymptotic estimate of \eqref{eqq3.47}. Thus the proof is complete.
	\end{proof}
	
	\begin{lem}\label{LE3.3}
		Assume that hypothesis {\rm (H)} is verified.
		If $\|U\|_{\mathcal{H}}=o\left(1\right)$,  on  $(\alpha,\beta)\subset(0,L)$ then $\|U\|_{\mathcal{H}}=o\left(1\right)$  on $(0,L).$
	\end{lem}
	\begin{proof} From Lemma \ref{LE3.2}, we have $\|v^2_x\|=o\left(1\right)$ and $\|\lambda v^2\|=o\left(1\right)$. Therefore, we only need to prove the same results for $v^1$ and $v^3$. Let $\phi\in H^1_0\left(0,L\right)$ be a given function. We proceed the proof in two steps.
		\begin{enumerate}
			\item Multiplying equation \eqref{eqq3.37} by $2\phi  \overline{v^1_x}$ in $L^2(0,L)$ and use Dirichlet boundary conditions to get
			\begin{equation}
			\begin{array}{c}
			\displaystyle{-\rho_1\int_0^{L}\phi'\left|\lambda v^1\right|^2dx
				-k_1\int_0^{L}\phi'\left| v^1_x\right|^2dx}
			\\  \\
			\displaystyle{+2\Re\left\{
				k_1\int_0^{L}\phi v^2_x\overline{v^1_x}dx+\mathrm{l}\left(k_1+k_3\right) \int_0^{L}\phi v^3_x\overline{v^1_x}dx-\mathrm{l}^2k_3\int_0^{L}\phi v^1\overline{v^1_x}dx\right\}}\\ \\
			=\displaystyle{-2\rho_1 \Re\left\{\int_0^{L}\phi f^4\overline{v^1_x}dx-\mathit{i}\int_0^{L} \left( \phi'f^1+ \phi f^1_x\right)\lambda 
				\overline{ v^1}dx\right\}}. \label{eqq3.42p}
			\end{array}
			\end{equation}
			From \eqref{eqq3.32} and  \eqref{eqq3.34}-\eqref{eqq3.36}, we remark that
			\begin{equation}\label{eqq3.42}
			\left\|v^1\right\|=O\left(\frac{1}{\lambda}\right),\ \left\|v^2\right\|=O\left(\frac{1}{\lambda}\right),\ \left\|v^3\right\|=O\left(\frac{1}{\lambda}\right).
			\end{equation}
			Then, using equation \eqref{eqq3.42}, Lemma \ref{LE3.2} and the facts that $v^1_x$, $\lambda v^1$ are uniformly bounded in $L^2(0,L)$, $f^1$ converges to zero in $H_0^1(0,L)$, $f^4$ converges to zero in $L^2(0,L)$ in \eqref{eqq3.42p}, we get
			\begin{equation}\label{eqq3.52}
			-\rho_1\int_0^{L}\phi'\left|\lambda v^1\right|^2dx-k_1\int_0^{L}\phi'\left| v^1_x\right|^2dx+2\mathrm{l}\left(k_1+k_3\right)\Re\left\{\int_0^{L}\phi v^3_x\overline{v^1_x}dx\right\}=o\left(1\right).
			\end{equation}
			Similarly, multiplying equation \eqref{eqq3.39} by $2\phi  \overline{v^3_x}$ in $L^2(0,L)$,  we get
			\begin{equation}\label{eqq3.53}
			-\rho_1\int_0^{L}\phi'\left|\lambda v^3\right|^2dx-k_3\int_0^{L}\phi'\left| v^3_x\right|^2dx-2\mathrm{l}\left(k_1+k_3\right) \Re\left\{\int_0^{L}\phi v^1_x\overline{v^3_x}dx\right\}=o\left(1\right).
			\end{equation}
			Adding \eqref{eqq3.52} and \eqref{eqq3.53}, we get
			\begin{equation}\label{eqq3.54}
			\rho_1\int_0^{L}\phi'\left(\left|\lambda v^1\right|^2+\left|\lambda v^3\right|^2\right)dx+k_1\int_0^{L}\phi'\left| v^1_x\right|^2dx+k_3\int_0^{L}\phi'\left| v^3_x\right|^2dx=o\left(1\right).
			\end{equation}
			\item Let $\epsilon>0$ such that $\alpha+\epsilon<\beta$ and define the cut-off function $\varsigma_1 \text{ in } C^1\left(\left[0,L\right]\right)$ by
			$$0\leq \varsigma_1 \leq 1,\ \varsigma_1=1 \text{ on } \left[0,\alpha\right] \text{ and } \varsigma_1=0 \text{ on } \left[\alpha+\epsilon,L\right].$$
			Take $\phi=x  \varsigma_1$ in \eqref{eqq3.54} and use the fact that $\left\|U\right\|_{\mathcal{H}}=o\left(1\right)$ on $\left(\alpha,\beta\right)$, we get
			\begin{equation}\label{eqq3.55}
			\rho_1\int_0^{\alpha}\left|\lambda v^1\right|^2dx+\rho_1\int_0^{\alpha}\left|\lambda v^3\right|^2dx+k_1\int_0^{\alpha}\left| v^1_x\right|^2dx+k_3\int_0^{\alpha}\left| v^3_x\right|^2dx=o\left(1\right).
			\end{equation}
			Using Lemmas \ref{LE3.1} and \ref{LE3.2}, in   \eqref{eqq3.55}, we get
			$$\left\|U\right\|_{\mathcal{H}}=o\left(1\right)\text{ on }(0,\alpha).$$
			Similarly, by symmetry, we can prove that
			$\left\|U\right\|_{\mathcal{H}}=o\left(1\right)\text{ on }(\beta,L)$ and
			therefore
			$$\left\|U\right\|_{\mathcal{H}}=o\left(1\right)\text{ on }(0,L).$$
		\end{enumerate}
		Thus the proof is complete.
	\end{proof}
	
	In the sequel, let $0<\alpha < \beta < L$ and consider the function $\varsigma\in C^1\left(\left[0,L\right]\right)$ such that $0\leq \varsigma\leq 1,\ \varsigma=1$ on $\left[\alpha+\epsilon,\beta-\epsilon\right]\subset \left[0,L\right]$ and $\varsigma=0$ on $\left[0,\alpha\right]\cup \left[\beta,L\right]$. Our aim is to prove that $\left\|U\right\|_{\mathcal{H}}=o\left(1\right)$ on $[\alpha,\beta]$ and so by Lemma \ref{LE3.3}, we get $\left\|U\right\|_{\mathcal{H}}=o\left(1\right)$ on $(0,L)$ contradicting \eqref{eqq3.32}.
	
	\begin{lem}\label{LE3.4}
		Suppose that hypothesis {\rm (H)} and   \eqref{eqq3.1} are satisfied. Then we have
		\begin{equation}\label{eqq3.57}
		\int_0^{L}\varsigma\left|v^1_x\right|^2dx=o\left(1\right),\ \text{ and } \int_0^{L}\varsigma\left|\lambda v^1\right|^2dx=o\left(1\right).
		\end{equation}
	\end{lem}
	\begin{proof}
		We show the first estimation of \eqref{eqq3.57}. We proceed in two main steps.
		\begin{enumerate}
			\item  Our first aim is to show that
			\begin{equation}\label{eqq3.63}
			\begin{array}{ll}
			\dfrac{k_1}{k_2}\intspace\varsigma\left| v^1_x\right|^2dx+ \left( \dfrac{\rho_2}{k_2}-\dfrac{\rho_1}{k_1}\right) \Re\left\lbrace \intspace\lambda^2 v^2_x\varsigma \overline{v^1}dx\right\rbrace\\ \\
			+\Re\left\lbrace \dfrac{\rho_1\lambda^2}{k_1k_2}\intspace\left( g^0 v^2_{x}-\inttime g\left(s\right)v^7_{x}ds\right) \varsigma \overline{v^1}dx\right\rbrace =o\left(1\right).
			\end{array}
			\end{equation}
			Multiplying \eqref{eqq3.38} by $\varsigma \overline{ v^1_x}$  in $L^2(0,L)$ and using by parts integration. Then, using Lemmas \ref{LE3.1}, \ref{LE3.2} and the facts that $v^1_x$, $\lambda v^1$ are uniformly bounded in $L^2(0,L)$, $f^2$ converges to zero in $H_0^1(0,L)$, $f^5$ converges to zero in $L^2(0,L)$, we get
			\begin{equation}\label{eqq3.58}
			\begin{array}{ll}
			k_1\intspace\varsigma\left| v^1_x\right|^2dx+  \rho_2\lambda^2\intspace\varsigma  v^2_x\overline{v^1}dx\\\\
			+\intspace \left( \widetilde{k}_2  v^2_x+\inttime g\left(s\right)v^7_{x}ds\right)\varsigma \overline{v^1_{xx}}dx =o\left(1\right).
			\end{array}
			\end{equation}
			Furthermore, multiplying \eqref{eqq3.37} by $\frac{\varsigma}{k_1}\left(\widetilde{k}_2\overline{v^2_{x}}+\int_0^{+\infty} g\left(s\right)\overline{v^7_{x}}ds\right)$ in $L^2(0,L)$ and using by parts integration. Then, using Lemmas \ref{LE3.1}, \ref{LE3.2} and the facts that $v^3_x$, $\lambda v^1$ are uniformly bounded in $L^2(0,L)$, $f^1$ converges to zero in $H_0^1(0,L)$, $f^4$ converges to zero in $L^2(0,L)$, we get
			\begin{equation}\label{eqq3.59}
			\begin{array}{ll}
			\dfrac{\rho_1\lambda^2}{k_1}\intspace\left( \widetilde{k}_2\overline{v^2_{x}}+\inttime g\left(s\right)\overline{v^7_{x}}ds\right) \varsigma v^1dx+ \intspace\left( \widetilde{k}_2 \overline{v^2_{x}}+\inttime g\left(s\right)\overline{v^7_{x}}ds\right)  \varsigma v^1_{xx}dx\\ \\
			+\dfrac{\mathit{i}\rho_1}{k_1}\intspace\left( \lambda\inttime g\left(s\right) \overline{v^7_{x}}ds\right) \varsigma f^1dx=o\left(1\right).
			\end{array}
			\end{equation}
			Subtracting  \eqref{eqq3.58} from \eqref{eqq3.59} and take the real part of the resulting equation, we get 
			\begin{equation}\label{N321}
			\begin{array}{ll}
			k_1\displaystyle\intspace\varsigma\left| v^1_x\right|^2dx -\Re\left\lbrace \dfrac{\rho_1\lambda^2}{k_1}\intspace\left( \widetilde{k}_2\overline{v^2_{x}}+\inttime g\left(s\right)\overline{v^7_{x}}ds\right) \varsigma v^1dx\right\rbrace  \\ \\
			+\Re\left\lbrace  \rho_2\lambda^2\intspace\varsigma v^2_{x}
			\overline{v^1}dx - \dfrac{\mathit{i}\rho_1}{k_1}\intspace\left( \lambda\inttime g\left(s\right) \overline{v^7_{x}}ds\right) \varsigma f^1dx\right\rbrace =o\left(1\right)
			.
			\end{array}
			\end{equation} 
			From \eqref{eqq3.40}, we have
			\begin{equation}\label{eqq3.60}
			\lambda v^7_x-\mathit{i}v^7_{xs}-\lambda  v^2_x =-ih^7_x,\ \text{ in }\ L^2_g\left(\mathbb{R}_+,L^2\right).
			\end{equation}
			Multiplying \eqref{eqq3.60} by $\varsigma \overline{ f^1}$ in
			$L^2_g\left(\mathbb{R}_+,L^2\right)$ and using by parts integration. Then, using hypothesis (H), Lemmas \ref{LE3.1}, \ref{LE3.2} and the facts that $f^1$, $f^2$ converge to zero in $H_0^1(0,L)$, $f^7$ converges to zero in $L_g^2(\rr_+;H^1_0(0,L))$, we get
			\begin{equation}\label{N123}
			\begin{array}{ll}
			\intspace\left( \lambda\inttime g\left(s\right)v^7_{x}ds\right)  \varsigma \overline{ f^1}dx
			=-\mathit{i} \intspace\left( \inttime g^\prime\left(s\right)v^7_{x}ds\right)  \varsigma \overline{f^1}dx\\\\
			-g^0 \intspace \lambda  v^2\left(\varsigma'  \overline{f^1}+\varsigma \overline{f^1_x}\right)dx + ig^0 \intspace  f^2_x\varsigma \overline{f^1}dx-i\intspace\inttime g(s)f^7_x\varsigma \overline{f^1}dx=o\left(1\right).
			\end{array}
			\end{equation}
			Finally, inserting   \eqref{N123} in \eqref{N321} and using the fact that $\tilde{k}_2=k_2-g^0,$ we get \eqref{eqq3.63}.
			\item Our next aim is to prove
			\begin{equation}\label{eqq3.60'}
			k_1\int_0^{L}\varsigma\left| v^1_x\right|^2dx+\lambda^2k_2\left( \frac{\rho_2}{k_2}
			-\frac{\rho_1}{k_1}\right)\int_0^{L}\varsigma  v^2_x\overline{v^1}dx=o\left(1\right).
			\end{equation}
			Multiplying \eqref{eqq3.60} by $\displaystyle{\frac{\rho_1}{k_1}\lambda\varsigma \overline{v^1}}$ in
			$L^2_g\left(\mathbb{R}_+,L^2\right)$ and using by parts integration. Then, using hypothesis (H), Lemma \ref{LE3.1}, and the facts that $\lambda v^1$ is uniformly bounded in $L^2(0,L)$, $f^2$ converges to zero in $H_0^1(0,L)$, $f^7$ converges to zero in $L_g^2(\rr_+;H^1_0(0,L))$, we get
			\begin{equation}\label{eqq3.65}
			\begin{array}{ll}
			\dfrac{\rho_1\lambda^2}{k_1k_2}  \intspace\inttime g\left(s\right)v^7_{x}\varsigma \overline{ v^1}dsdx
			=-i \dfrac{\rho_1}{k_1k_2}\intspace\inttime g^\prime\left(s\right)v^7_{x}\varsigma \lambda \overline{v^1}dsdx\\\\
			+\dfrac{\rho_1\lambda^2}{k_1k_2}g^0\intspace v^2_x\varsigma \overline{v^1}dx
			-i \dfrac{\rho_1}{k_1k_2}\intspace\inttime g\left(s\right)(f^7_{x}-f^2_x)\varsigma \lambda \overline{v^1}dsdx
			=o\left(1\right).
			\end{array}
			\end{equation}
			Adding \eqref{eqq3.63} and  \eqref{eqq3.65},  we deduce \eqref{eqq3.60'}.
			\item Finally, using condition \eqref{eqq3.1} in \eqref{eqq3.60'}, we  get the first estimation of \eqref{eqq3.57}. Moreover, multiplying \eqref{eqq3.37} by $\varsigma \overline{v^1}$  in $L^2\left(0,L\right)$ , using \eqref{eqq3.32}, \eqref{eqq3.33}, \eqref{eqq3.42}, and the first estimation of \eqref{eqq3.57},
			we can easily prove that $$ \int_0^{L}\varsigma\left| \lambda v^1\right|^2dx=o\left(1\right).$$
		\end{enumerate}
		Thus the proof is complete.
	\end{proof}
	\begin{lem}\label{LE3.5}
		Assume that hypothesis {\rm (H)} and \eqref{eqq3.1} are satisfied. Then
		\begin{equation}\label{eqq3.67}
		\int_0^{L}\varsigma\left|v^3_x\right|^2dx=o\left(1\right)\,\,\, {\rm and }\,\,\,  \int_0^{L}\varsigma\left| \lambda v^3\right|^2dx=o\left(1\right).
		\end{equation}
	\end{lem}
	\begin{proof}
		Multiplying \eqref{eqq3.37} by $\varsigma \overline{v_x^3}$ in $L^2(0,L)$ and using by parts integration. Then, using Lemmas \ref{LE3.2}, \ref{LE3.4} and the facts that  $f^1$ converges to zero in $H_0^1(0,L)$, $f^4$ converges to zero in $L^2(0,L)$, we get
		\begin{equation}\label{eqq3.68}
		\rho_1\intspace\lambda^2 v^1 \varsigma \overline{v_x^3}dx+\mathrm{l}\left( k_1+k_3\right)\intspace\varsigma\left|v^3_x\right|^2dx -k_1\intspace v^1_{x}\varsigma \overline{v_{xx}^3}dx=o\left(1\right).
		\end{equation}
		Moreover,  multiplying \eqref{eqq3.39} by $\varsigma\overline{ v^1_x}$ in $L^2(0,L)$ and using by parts integration. Then, using Lemmas \ref{LE3.2}, \ref{LE3.4} and the facts that  $\lambda v^3$ is uniformly bounded in $L^2(0,L)$, $f^3$ converges to zero in $H_0^1(0,L)$, $f^6$ converges to zero in $L^2(0,L)$, we get
		\begin{equation}\label{eqq3.69}
		-\rho_1\intspace \lambda^2 v^3_x  \varsigma \overline{v^1} dx+ k_3\intspace \overline{v^1_x}\varsigma v^3_{xx}dx=o\left(1\right).
		\end{equation}
		Take the real part of the sum of \eqref{eqq3.68} and \eqref{eqq3.69}. Then, using the fact that $k_1=k_3$, we get
		\begin{equation}\label{eqq3.70}
		\int_0^{L}\varsigma\left|v^3_x\right|^2dx=o\left(1\right).
		\end{equation}
		Next, if we multiplying \eqref{eqq3.39} by $\varsigma\overline{ v^3}$ in $L^2\left(0,L\right)$, then from  \eqref{eqq3.42}, \eqref{eqq3.70} and Lemma \ref{LE3.4}, we deduce that
		$$\rho_1\int_0^{L}\varsigma\left|\lambda v^3\right|^2dx=o\left(1\right).$$
		Thus the proof is complete.
	\end{proof}
	
	\noindent \textbf{Proof of Theorem \ref{TH3.5}} \ Using Lemma \ref{LE3.1}, Lemma \ref{LE3.2}, Lemma \ref{LE3.4}, and  Lemma \ref{LE3.5},  we get
	$\|U\|_{\mathcal{H}}=o\left(1\right)$ on  $\left[\alpha+\epsilon,\beta-\epsilon\right]$. Hence,  by Lemma \ref{LE3.3},  we get $\|U\|_{\mathcal{H}}=o\left(1\right)$ on $\left[0,L\right]$ which contradicts \eqref{eqq3.32}. Therefore, {\rm (H2)} holds and so, by \cite{Huang:85} and \cite{Pruss}, we deduce the  exponential stability of the system \eqref{eqq2.1}-\eqref{eqq2.3} propagating with equal speeds.
	
	\begin{rk}
		It is easy to see that our technique used for the proof of the exponential stability of the Bresse system under fully Dirichlet boundary conditions is also valid under mixed boundary conditions. 
	\end{rk}
	
	\subsection{Lack of exponential  stability with  different speed}\label{se4}
	
	\noindent In this part, our goal is to show that the elastic Bresse system \eqref{eqq2.1}-\eqref{eqq2.3}  with fully Dirichlet boundary conditions  is not exponentially stable if the speeds of propagation of the waves are different. In particular, we consider the case when $\mathrm{l}\rightarrow0$; i.e, when \eqref{eqq2.1}-\eqref{eqq2.3} reduces to the  Timoshenko system \eqref{timo}-\eqref{dirich} with $\frac{\rho_1}{k_1}\neq\frac{\rho_2}{k_2}$.
	In fact, when the speeds of propagation are different, if mixed Dirichlet-Neumann boundary conditions  are considered in  system \eqref{eqq2.1}  instead of fully Dirichlet boundary conditions, then we can easily show that the system is not exponentially decaying. Indeed,  similar to \cite{AlabauBresse}, \cite{Wehbenadine}, \cite{Lucifatori}, \cite{hugo-Sare-racke}, and \cite{santossoufyane}, the idea is to find a sequence of $(\lambda_n)_n\subseteq \mathbb{R}$ with $\left| \lambda_n\right| \longrightarrow +\infty$ and a sequence of vectors  $(U_n)_n\subseteq D\left(\mathcal{A}\right)$ with $\left\| U_n\right\| _\HH=1$ such that $\left\| (i\lambda_n Id-\mathcal{A})U_n\right\|_\HH\longrightarrow 0$. In the case of  Dirichlet-Neumann-Neumann boundary condition, this approach worked well due to the fact that all eigenmodes are separable, i.e., the system operator can be decomposed to a block-diagonal form according to the frequency when the state variables are expanded into Fourier series. However, in the case of fully Dirichlet boundary conditions,
	this approach has no success in the literature to our knowledge and the problem is still be open. Consequently, in this section, we use  another approach based on the behavior of the spectrum  to prove the lack of exponential stability of the system mainly in the case when $\mathrm{l}\rightarrow0$. For simplicity, in this section, we take $L=1$ so \eqref{eqq2.1}-\eqref{eqq2.3}  reduces to the following  Timoshenko system:
	
	\begin{equation}\label{timo}
	\left\{
	\begin{array}{l}
	\displaystyle{
		\rho_1\varphi_{tt}-k_1 \left(\varphi_x+\psi\right)_x=0,}\\
	\displaystyle{\rho_2 \psi_{tt}-\widetilde{k}_2\psi_{xx}+k_1\left(\varphi_x+\psi\right)-\int_0^{+\infty} g\left(s\right)\eta_{xx}ds=0,}\\
	\displaystyle{\eta_t+\eta_s-\psi_t=0,}
	\end{array}
	\right.
	\end{equation}
	with the  initial conditions
	\begin{equation*}
	\begin{array}{lll}
	\varphi\left(\cdot,0\right)=\varphi_0\left(\cdot\right),\ \psi\left(\cdot,-t\right)=\psi_0\left(\cdot,t\right),
	& \\
	\varphi_t\left(\cdot,0\right)=\varphi_1\left(\cdot\right),\
	\psi_t\left(\cdot,0\right)=\psi_1\left(\cdot\right),\
	& \\
	\eta^0\left(\cdot,s\right):=\eta\left(\cdot,0,s\right)=\psi_0\left(\cdot,0\right)-\psi_0\left(\cdot,s\right)&\text{in }\left(0,1\right),\ s\geq0,
	\end{array}
	\end{equation*}
	and fully Dirichlet boundary conditions
	\begin{equation}\label{dirich}
	\begin{array}{lll}
	\varphi\left(0,\cdot\right)=\varphi\left(1,\cdot\right)=\psi\left(0,\cdot\right)=\psi\left(1,\cdot\right)=0\quad&\text{in }\mathbb{R}_{+},\\
	\eta\left(0,\cdot,\cdot\right)=\eta\left(1,\cdot,\cdot\right)=0\quad&\text{in }\mathbb{R}_{+}\times\mathbb{R}_{+},\\
	\eta\left(\cdot,\cdot,0\right)=0&\text{in} \left(0,1\right)\times\mathbb{R}_{+}.
	\end{array}
	\end{equation}
	In this case, the energy space $\mathcal{H}$ reduces to
	\begin{equation*}
	\mathcal{H}_1=\left(H_0^1\left(0,1\right)\right)^2\times\left(L^2\left(0,1\right)\right)^2\times L^2_g\left(\mathbb{R}_+,H^1_0\right)
	\end{equation*}
	and the generator $\mathcal{A}$ becomes the operator $\mathcal{A}_1$  defined by
	\begin{equation*}
	\begin{array}{l}
	D\left(\mathcal{A}_1\right)=\bigg\{\ U=(v^1,v^2,v^3,v^4,v^5)^{\mathsf{T}}\in\mathcal{H}_1 \ |\  v^1\in H^2\left(0,1\right),\  v^3,v^4\in  H^1_0\left(0,1\right),\ v^5_s\in L^2_g\left(\mathbb{R}_+,H^1_0\right), \\ \hspace{3.5cm} \ v^2+\int_0^{+\infty}g\left(s\right)v^5ds\in H^2\left(0,1\right)\cap H^1_0\left(0,1\right),\ v^5\left(x,0\right)=0\bigg\}
	\end{array}
	\end{equation*}
	and
	\begin{equation}\label{generator}
	\mathcal{A}_1U=\left(\begin{array}{c}
	v^3\\ v^4\\
	\rho_1^{-1}k_1 \left(v^1_x+v^2\right)_x\\
	\rho_2^{-1} \left(\widetilde{k}_2 v^2_{xx}-k_1\left(v^1_x+v^2\right)+\int_0^{+\infty} g\left(s\right)v^5_{xx}ds\right)\\
	v^2-v^5_s
	\end{array}\right)
	\end{equation}
	for all $U=\left(v^1,v^2,v^3,v^4,v^5\right)^{\mathsf{T}}\in D\left(\mathcal{A}_1\right).$\\
	
	Throughout this part, in addition to hypothesis {\rm (H)}, we assume that
	\begin{equation}\tag{$\text{H}'$}
	\frac{\rho_1}{k_1}\neq\frac{\rho_2}{k_2} \  \textrm {and}  \ \left|g''\left(s\right)\right|\leq c_2 g\left(s\right) \ \textrm{for some } c_2>0.
	\end{equation}
	\begin{thm}\label{nonuniform}
		Under hypothesis {\rm (H)} and ({$\text{H}'$}), system \eqref{timo}-\eqref{dirich} is not uniformly stable in the energy space $\mathcal{H}_1.$
	\end{thm}
	For the proof of Theorem \ref{nonuniform}, we aim to  show that an infinite number of eigenvalues of $\mathcal{A}_1$ approach the imaginary axis which prevents the Timoshenko system \eqref{timo}-\eqref{dirich} from being exponentially stable.  First we determine the characteristic equation satisfied by the eigenvalues of $\mathcal{A}_1$. For this aim, Let $\lambda\in\mathbb{C}$ be an eigenvalue of $\mathcal{A}_1$ and let $U=\left(v^1,v^2,v^3,v^4,v^5\right)^{\mathsf{T}}\in D(\mathcal{A}_1)$  be an associated eigenvector such that $\|U\|_{\mathcal{H}_1}=1$. Then
	\begin{eqnarray}
	v^3=\lambda v^1 \label{1.1},\\
	v^4=\lambda v^2 \label{1.2},\\
	k_1\left(v^1_x+v^2\right)_x=\rho_1 \lambda v^3\label{1.3},\\
	\widetilde{k}_2 v^2_{xx}-k_1\left(v^1_x+v^2\right)+\int_0^{+\infty}g\left(s\right)v^5_{xx} ds=\rho_2 \lambda v^4\label{1.4},\\
	v^4-v^5_s=\lambda v^5\label{1.5}.
	\end{eqnarray}
	From \eqref{1.5} and \eqref{1.2}, we have
	$$v^5_s+\lambda v^5=\lambda v^2.$$
	Integrating this equation and using the fact that $v^5\left(x,0\right)=0$, we get
	\begin{equation}\label{1.6}
	v^5=v^2\left(1-e^{-\lambda s}\right).
	\end{equation}
	Inserting \eqref{1.6}, \eqref{1.1}-\eqref{1.2} in \eqref{1.3}-\eqref{1.4}, we get
	\begin{eqnarray}
	\frac{k_1}{\rho_1}\left(v^1_x+v^2\right)_x= \lambda^2 v^1\label{1.7},\\
	\frac{\underline{k}_2}{\rho_2} v^2_{xx}-\frac{k_1}{\rho_2}\left(v^1_x+v^2\right)= \lambda^2 v^2\label{1.8},
	\end{eqnarray}
	where $\displaystyle{\underline{k}_2=k_2-\int_0^{+\infty}g\left(s\right)e^{-\lambda s}ds}$. Equivalently, we have
	\begin{equation}\label{1.13}
	\left\{
	\begin{array}{ll}
	\displaystyle{v^2_{xxxx}-\left( \frac{\rho_2}{\underline{k}_2} +\frac{\rho_1}{k_1}\right)\lambda^2 v^2_{xx}+\frac{\rho_1\rho_2}{k_1\underline{k}_2}\lambda^2 \left(\lambda^2+\frac{k_1}{\rho_2}\right) v^2= 0,}\\  \\
	\displaystyle{v^2\left(\zeta\right)=0,\ v^2_{xxx}\left(\zeta\right)- \frac{\rho_2}{\underline{k}_2}\lambda^2 v^2_x\left(\zeta\right)=0,\ \zeta=0,1.}
	\end{array}\right.
	\end{equation}
	The solution of \eqref{1.13} is given by
	$$v^2\left(x\right)=\sum_{j=1}^{4}c_j e^{r_j x},$$
	where $c_j\in\mathbb{C}$ for all $1\leq j\leq4$ and
	\begin{equation*}
	\left\{
	\begin{array}{l}
	\displaystyle{r_1=\lambda\sqrt{\dfrac{\left( \frac{\rho_2}{\underline{k}_2} +\frac{\rho_1}{k_1}\right)+\sqrt{\left( \frac{\rho_2}{\underline{k}_2} -\frac{\rho_1}{k_1}\right)^2-\frac{4\rho_1}{\underline{k}_2\lambda^2}}}{2}},\ r_2=-r_1,}\\ \\
	\displaystyle{r_3=\lambda\sqrt{\dfrac{\left( \frac{\rho_2}{\underline{k}_2} +\frac{\rho_1}{k_1}\right)-\sqrt{\left( \frac{\rho_2}{\underline{k}_2} -\frac{\rho_1}{k_1}\right)^2-\frac{4\rho_1}{\underline{k}_2\lambda^2}}}{2}},\ r_4=-r_3.}
	\end{array}\right.
	\end{equation*}
	The boundary conditions in \eqref{1.13} can be expressed by  $$M C=0$$
	where
	\begin{equation*}
	M=\begin{pmatrix}
	1&1&1&1\\
	e^{r_1}&e^{-r_1}&e^{r_3}&e^{-r_3}\\
	f\left(r_1\right)&-f\left(r_1\right)&f\left(r_3\right)&-f\left(r_3\right)\\
	f\left(r_1\right)e^{r_1}&-f\left(r_1\right)e^{-r_1}
	&f\left(r_3\right)e^{r_3}&-f\left(r_3\right)e^{-r_3}
	\end{pmatrix},\ C=\begin{pmatrix}
	c_1\\c_2\\c_3\\c_4
	\end{pmatrix},
	\end{equation*}
	and $f\left(r\right)=r^3-\frac{\rho_2}{\underline{k}_2} r \lambda^2.$ For shortness,  denote by  $f\left(r_1\right)=f_1$ and $f\left(r_3\right)=f_3.$ Then
	\begin{equation}\label{1.16}
	\left\{
	\begin{array}{ll}
	\displaystyle{f_1=\dfrac{r_1\lambda^2}{2}\left[\left(\frac{\rho_1}{k_1}- \frac{\rho_2}{\underline{k}_2} \right)+\sqrt{\left( \frac{\rho_2}{\underline{k}_2} -\frac{\rho_1}{k_1}\right)^2-\frac{4\rho_1}{\underline{k}_2\lambda^2}}\right],}\\ \\
	\displaystyle{f_3=\dfrac{r_3\lambda^2}{2}\left[\left(\frac{\rho_1}{k_1}- \frac{\rho_2}{\underline{k}_2} \right)-\sqrt{\left( \frac{\rho_2}{\underline{k}_2} -\frac{\rho_1}{k_1}\right)^2-\frac{4\rho_1}{\underline{k}_2\lambda^2}}\right],}
	\end{array}\right.
	\end{equation}
	and
	\begin{equation}\left\{
	\begin{array}{ll}
	\displaystyle{\left(f_1+f_3\right)^2=\lambda^6\frac{\rho_1}{k_1}\left(\frac{\rho_1}{k_1}- \frac{\rho_2}{\underline{k}_2} \right)^2-\lambda^4\frac{\rho_1}{\underline{k}_2}\left(\frac{3\rho_1}{k_1}- \frac{\rho_2}{\underline{k}_2} \right)+2\lambda^4\frac{\rho_1}{\underline{k}_2}\sqrt{\frac{\rho_1\rho_2}{k_1\underline{k}_2} \left(1+\frac{k_1}{\rho_2\lambda^2}\right)}\ ,}\\ \\
	\displaystyle{\left(f_1-f_3\right)^2=\lambda^6\frac{\rho_1}{k_1}\left(\frac{\rho_1}{k_1}- \frac{\rho_2}{\underline{k}_2} \right)^2-\lambda^4\frac{\rho_1}{\underline{k}_2}\left(\frac{3\rho_1}{k_1}- \frac{\rho_2}{\underline{k}_2} \right)-2\lambda^4\frac{\rho_1}{\underline{k}_2}\sqrt{\frac{\rho_1\rho_2}{k_1\underline{k}_2} \left(1+\frac{k_1}{\rho_2\lambda^2}\right)}\ .}
	\end{array}\right. \label{f1f2}
	\end{equation}
	Therefore, using \eqref{1.16} and \eqref{f1f2}, we get 
	\begin{equation*}
	\begin{array}{llll}
	\displaystyle{det\left(M\right)}&=& -2\left(f_1-f_3\right)^2\cosh\left(r_1+r_3\right)+2\left(f_1+f_3\right)^2\cosh\left(r_1-r_3\right)-8 f_1 f_3\\ \\
	&=&\displaystyle{-4\lambda^6\frac{\rho_1}{k_1}\left(\frac{\rho_1}{k_1}- \frac{\rho_2}{\underline{k}_2} \right)^2\sinh\left(r_1\right)\sinh\left(r_3\right)+4\lambda^4\frac{\rho_1}{\underline{k}_2}\left(\frac{3\rho_1}{k_1}- \frac{\rho_2}{\underline{k}_2} \right)\sinh\left(r_1\right)\sinh\left(r_3\right)}\\ \\
	&&\displaystyle{-8\lambda^4\frac{\rho_1}{\underline{k}_2}\sqrt{\frac{\rho_1\rho_2}{k_1\underline{k}_2} \left(1+\frac{k_1}{\rho_2\lambda^2}\right)}\cosh\left(r_1\right)\cosh\left(r_3\right)-8\lambda^4\frac{\rho_1}{\underline{k}_2}\sqrt{\frac{\rho_1\rho_2}{k_1\underline{k}_2} \left(1+\frac{k_1}{\rho_2\lambda^2}\right)}\ .}
	\end{array}
	\end{equation*}
	Equation \eqref{1.13}  admits a non trivial solution if and only if $\displaystyle{det\left(M\right)}=0$; i.e, if and only if the
	eigenvalues of $\mathcal{A}_1$ are roots of  
	the function $F$ defined by:
	\begin{equation} \label{characeq}
	\begin{array}{lll}
	F(\lambda)=& \displaystyle{\frac{\lambda^6}{k_1}\left(\frac{\rho_1}{k_1}- \frac{\rho_2}{\underline{k}_2} \right)^2\sinh\left(r_1\right)\sinh\left(r_3\right)-\frac{\lambda^4}{\underline{k}_2}\left(\frac{3\rho_1}{k_1}- \frac{\rho_2}{\underline{k}_2} \right)\sinh\left(r_1\right)\sinh\left(r_3\right)}\\ \\
	&+\displaystyle{\frac{2\lambda^4}{\underline{k}_2}\sqrt{\frac{\rho_1\rho_2}{k_1\underline{k}_2} \left(1+\frac{k_1}{\rho_2\lambda^2}\right)}\cosh\left(r_1\right)\cosh\left(r_3\right)+\frac{2\lambda^4}{\underline{k}_2}\sqrt{\frac{\rho_1\rho_2}{k_1\underline{k}_2} \left(1+\frac{k_1}{\rho_2\lambda^2}\right)}\ }.
	\end{array}
	\end{equation}
	\begin{lem}\label{realpartbdd}
		Let $\lambda\in\mathbb{C}$ be an eigenvalue of $\mathcal{A}_1$. Then $\Re(\lambda)$ is bounded.
	\end{lem}
	\begin{proof}
		Multiplying \eqref{1.7} and \eqref{1.8} by $-\rho_1\overline{\varphi}$, and $-\rho_2\overline{\psi}$ respectively, and integrating their sum, we get
		\begin{equation*}
		\rho_1\left\|\lambda\varphi\right\|^2+\rho_2\left\|\lambda\psi\right\|^2+ k_1\left\|\varphi_x+\psi\right\|^2+k_2\left\|\psi_x\right\|^2-\left\|\psi_x\right\|^2\int_0^{+\infty}g\left(s\right)e^{-\lambda s}ds=0.
		\end{equation*}
		Since  $\left\|U\right\|_{\mathcal{H}_1}=1$ then $\rho_1\left\|\lambda\varphi\right\|^2+\rho_2\left\|\lambda\psi\right\|^2+ k_1\left\|\varphi_x+\psi\right\|^2+k_2\left\|\psi_x\right\|^2$ and $\left\|\psi_x\right\|^2$ are bounded. Therefore
		\begin{equation}\label{n1.9}
		\int_0^{+\infty}g\left(s\right)e^{-\lambda s}ds<+\infty.
		\end{equation}
		Hence,
		$$\lim_{s\to +\infty}g\left(s\right)e^{-\Re(\lambda) s}=0.$$
		Since  $ \mathcal{A}_1$ is dissipative in $\mathcal{H}_1 $ then $\Re(\lambda)\leq0$ and consequentially there exists constant $a>0$ such that 
		$$
		-a\leq\Re(\lambda)<0
		$$
		and hence the proof is complete.
	\end{proof}
	\begin{prop} \label{asymp}
		Assume that hypothesis {\rm (H)} and ({$\text{H}'$}) are satisfied. Then, there exist $n_0,n'_0\in \mathbb{N}$ sufficiently large such that 
		\begin{equation}
		\sigma\left(\mathcal{A}_1\right)\supset \tilde{\sigma_0}\cup \tilde{\sigma_1},
		\end{equation}
		where $\tilde{\sigma_0}\cup \tilde{\sigma_1}$ is the set of eigenvalues of the operator $\mathcal{A}_1$ such that 
		\begin{equation}
		\tilde{\sigma_1}=\left\{\tilde{\lambda}_j^{(0)},\tilde{\lambda}_j^{(1)}\right\}_{j\in J},\quad \tilde{\sigma_0}=\left\{\lambda^{\left(0\right)}_n,\ \lambda^{\left(1\right)}_{n'}\right\}_{
			\begin{array}{ll}
			n,\ n'\in\mathbb{Z}\\
			\left|n\right|\geq n_0,\ \left|n'\right|\geq n_0'
			\end{array}
		},\quad
		\tilde{\sigma_0}\cap \tilde{\sigma_1}=\emptyset,
		\end{equation}
		where $J$ is a finite set. Moreover, $\lambda_n^{(0)}$ and $\lambda_n^{(1)}$ are simple and satisfies the following asymptotic behavior
		\begin{equation}\label{branch1}
		\displaystyle{\lambda^{\left(0\right)}_n=\mathit{i} n\pi\sqrt{\frac{k_2}{\rho_2}}-\frac{ g\left(0\right)}{2k_2}+o\left(1\right)}, \ \ \forall \ |n|\geq n_0
		\end{equation}
		and
		\begin{equation}\label{branch2}
		\displaystyle{\lambda^{\left(1\right)}_{n'}=\mathit{i} n'\pi\sqrt{\frac{k_1}{\rho_1}}+o\left(1\right)}, \ \ \forall\  |n'|\geq n'_0.
		\end{equation}
	\end{prop}
	\begin{proof} The proof is divided into three steps. Step 1 and Step 2 furnish an asymptotic development of the characteristic
		equation for large $\lambda$. Step 3 gives a limited development of the large eigenvalues $\lambda$. \\
		\noindent {\bf Step 1.} In this step, we prove the following asymptotic behavior estimate 
		\begin{equation}\label{kdash2}
		\frac{1}{\underline{k}_2}=\frac{1}{k_2}+\frac{g\left(0\right)}{k_2^2\lambda}+O\left(\frac{1}{\lambda^2}\right).\end{equation}
		Indeed, integration by parts yields
		\begin{equation}\label{ibp}\underline{k}_2=k_2-\int_0^{+\infty}g\left(s\right)e^{-\lambda s}ds=k_2-\frac{g\left(0\right)}{\lambda}-\frac{g'\left(0\right)}{\lambda^2}-\frac{1}{\lambda^2}\int_0^{+\infty}g''\left(s\right)e^{-\lambda s}ds. \end{equation}
		From hypothesis ($\text{H}'$), since  $\left|g''\left(s\right)\right|\leq c_2 g\left(s\right)$, then
		\begin{equation}\label{g"}
		\left|\int_0^{+\infty}g''\left(s\right)e^{-\lambda s}ds\right|\leq c_2\int_0^{+\infty}g\left(s\right)e^{\Re(\lambda) s}ds.
		\end{equation}
		on the other hand, since
		$$\int_0^1\int_0^{+\infty}g\left(s\right)\left|\eta_x\right|^2ds\ dx<+\infty,$$
		then, from \eqref{1.6} and \eqref{g"}, we get
		\begin{equation}\label{g"1}
		\int_0^{+\infty}g''\left(s\right)e^{-\lambda s}ds=O\left(1\right).
		\end{equation}
		Finally, \eqref{ibp} and \eqref{g"1} yield \eqref{kdash2}.\\
		
		\noindent {\bf Step 2.}	In this step, we furnish an asymptotic development of the function $F(\lambda)$ for large $\lambda$. Assume that $\frac{\rho_1}{k_1}\neq\frac{\rho_2}{k_2}$, then we have
		\begin{equation}\label{1.17}
		\left\{\begin{array}{l}
		\displaystyle{r_1=\lambda\sqrt{\frac{\frac{\rho_2}{k_2}+\frac{\rho_1}{k_1}+\left|\frac{\rho_2}{k_2}-\frac{\rho_1}{k_1}\right|}{2}}+\frac{\rho_2 g\left(0\right)}{2\sqrt{2}k_2^2}\frac{1+sign\left(\frac{\rho_2}{k_2}-\frac{\rho_1}{k_1}\right)}{\sqrt{\frac{\rho_2}{k_2}+\frac{\rho_1}{k_1}+\left|\frac{\rho_2}{k_2}-\frac{\rho_1}{k_1}\right|}}+O\left(\frac{1}{\lambda}\right),}\\ \\
		\displaystyle{r_3=\lambda\sqrt{\frac{\frac{\rho_2}{k_2}+\frac{\rho_1}{k_1}-\left|\frac{\rho_2}{k_2}-\frac{\rho_1}{k_1}\right|}{2}}+\frac{\rho_2g\left(0\right)}{2\sqrt{2}k_2^2}\frac{1-sign\left(\frac{\rho_2}{k_2}-\frac{\rho_1}{k_1}\right)}{\sqrt{\frac{\rho_2}{k_2}+\frac{\rho_1}{k_1}-\left|\frac{\rho_2}{k_2}-\frac{\rho_1}{k_1}\right|}}+O\left(\frac{1}{\lambda}\right),}
		\end{array}\right.
		\end{equation}
		where $sign\left(\frac{\rho_2}{k_2}-\frac{\rho_1}{k_1}\right)=\dfrac{\left|\frac{\rho_2}{k_2}-\frac{\rho_1}{k_1}\right|}{\frac{\rho_2}{k_2}-\frac{\rho_1}{k_1}}$.  If $sign\left(\frac{\rho_2}{k_2}-\frac{\rho_1}{k_1}\right)=1$,  then \eqref{1.17} is equivalent to
		\begin{equation}\label{1.18}
		\left\{\begin{array}{l}
		\displaystyle{r_1=\lambda\sqrt{\frac{\rho_2}{k_2}}+\frac{ g\left(0\right)}{2k_2}{\sqrt{\frac{\rho_2}{k_2}}}+O\left(\frac{1}{\lambda}\right),}\\ \\
		\displaystyle{r_3=\lambda\sqrt{\frac{\rho_1}{k_1}}+O\left(\frac{1}{\lambda}\right).}
		\end{array}\right.
		\end{equation}
		If $sign\left(\frac{\rho_2}{k_2}-\frac{\rho_1}{k_1}\right)=-1,$ then \eqref{1.17} is equivalent to
		\begin{equation}\label{1.19}
		\left\{\begin{array}{l}
		\displaystyle{r_1=\lambda\sqrt{\frac{\rho_1}{k_1}}+O\left(\frac{1}{\lambda}\right).}
		\\ \\
		\displaystyle{r_3=\lambda\sqrt{\frac{\rho_2}{k_2}}+\frac{ g\left(0\right)}{2k_2}{\sqrt{\frac{\rho_2}{k_2}}}+O\left(\frac{1}{\lambda}\right).}
		\end{array}\right.
		\end{equation}
		In the sequel, we suppose that \eqref{1.18} holds since the analysis follows similarly. Now, inserting \eqref{1.18} in \eqref{characeq}  and using   Lemma \ref{realpartbdd}, we get
		\begin{equation}\label{Flambda}
		F(\lambda)=\frac{\lambda^6}{k_1}\left(\frac{\rho_1}{k_1}- \frac{\rho_2}{k_2}\right)^2\sinh\left(r_1\right)\sinh\left(r_3\right)+O(\lambda^5).
		\end{equation}
		\noindent {\bf Step 3.} In this step, we perform a limited development of of the large eigenvalues of the operator $\AA_1$. Let $\lambda$ be a large eigenvalue of $\mathcal{A}_1$, then from \eqref{Flambda}, $\lambda$ is large root of the following asymptotic equation  
		\begin{equation}\label{aeh}
		h(\lambda)=h_0\left(\lambda\right)+O\left(\frac{1}{\lambda}\right)=0, 
		\end{equation}
		where $h_0\left(\lambda\right)=\sinh\left(r_1\right)\sinh\left(r_3\right).$ Now, we prove that
		\begin{equation*}
		h_0\left(\lambda\right)=0 \ \textrm{ if and only if } \ r_1=\mathit{i} n\pi\ \text{ and } r_3=\mathit{i} n'\pi,\quad n,n'\in\mathbb{Z}.
		\end{equation*}
		Indeed, Suppose that
		$$r_1=\mathit{i} n\pi,\ \text{and } r_3\neq\mathit{i} n'\pi,\quad n,n'\in\mathbb{Z}.$$
		Then
		\begin{equation*}
		M=\begin{pmatrix}
		1&1&1&1\\
		\left(-1\right)^n&\left(-1\right)^n&e^{r_3}&e^{-r_3}\\
		f_1&-f_1&f_3&-f_3\\
		\left(-1\right)^nf_1&-\left(-1\right)^nf_1
		&f_3e^{r_3}&-f_3e^{-r_3}
		\end{pmatrix}.
		\end{equation*}
		Using Gaussian elimination,   $M$ is equivalent to the following matrix, denoted by
		\begin{equation}\label{1.20}
		\widetilde{M}=\begin{pmatrix}
		1&1&1&1\\
		0&0&e^{r_3}-\left(-1\right)^n&e^{-r_3}-\left(-1\right)^n\\
		f_1&-f_1&f_3&-f_3\\
		0&0&f_3\left(e^{r_3}-\left(-1\right)^n\right)&-f_3\left(e^{-r_3}-\left(-1\right)^n\right)
		\end{pmatrix}.
		\end{equation}
		Hence,
		\begin{equation*}
		\left\{
		\begin{array}{rl}
		\left(e^{r_3}-\left(-1\right)^n\right) c_3+\left(e^{-r_3}-\left(-1\right)^n\right) c_4&=0\\
		f_3\left(e^{r_3}-\left(-1\right)^n\right)c_3-f_3\left(e^{-r_3}-\left(-1\right)^n\right)c_4&=0.
		\end{array}
		\right.
		\end{equation*}
		From \eqref{1.16}, we can check that $f_1\neq0$ and $f_3\neq0$ for $\lambda$ large enough. Since $r_3\neq\mathit{i} n'\pi$ for all $n'\in\mathbb{Z}$, then
		$$c_3=c_4=0.$$
		From \eqref{1.20},  we have
		\begin{equation*}
		\left\{
		\begin{array}{ll}
		c_1+c_2=0\\
		f_1c_1-f_1c_2=0.
		\end{array}
		\right.
		\end{equation*}
		Since $f_1\neq0$, we get
		$$c_1=c_2=0 \ \ \textrm{and} \ \ v^2=0$$
		which is a contradiction with $\|U\|_{\mathcal{H}_1}=1$. Similarly if
		$$r_1\neq\mathit{i} n\pi\ \text{and } r_3=\mathit{i} n'\pi,\quad n,n'\in\mathbb{Z}$$
		we get $v^2=0$. We conclude that 
		$$h_0(\lambda)=0 \,\,\, \Leftrightarrow \,\,\, r_1=\mathit{i} n\pi\ \text{and } r_3=\mathit{i} n'\pi,\quad n,n'\in\mathbb{Z}.$$ Then from asymptotic equation \eqref{1.18}, the large roots of $h_0$ satisfy the following asymptotic equations
		\begin{equation}\label{branch11}
		\displaystyle{\mu^{\left(0\right)}_n=\mathit{i} n\pi\sqrt{\frac{k_2}{\rho_2}}-\frac{ g\left(0\right)}{2k_2}+O\left(\frac{1}{n}\right)}, \ \ \forall \ |n|\geq n_0
		\end{equation}
		and
		\begin{equation}\label{branch22}
		\displaystyle{\mu^{\left(1\right)}_{n'}=\mathit{i} n'\pi\sqrt{\frac{k_1}{\rho_1}}+O\left(\frac{1}{n}\right)}, \ \ \forall\  |n'|\geq n'_0.
		\end{equation}
		Next, with the help of   Rouch\'{e}'s Theorem and using the asymptotic equation \eqref{aeh}, it is easy to see that the large roots of $h$, $\lambda_n^{(0)}$ and $\lambda_n^{(1)}$,  are closed to those of  $h_0$. Thus the proof is complete.\\\\
	\end{proof}
	

	\noindent\textbf{\emph{Proof of Theorem \ref{nonuniform}}} From Proposition \ref{asymp}, the operator $\AA_1$ has two branches of eigenvalues, the energy corresponding to the first branch $\lambda_n^{(0)}$ decays exponentially and the energy corresponding to the second branch of eigenvalues $\lambda_{n^\prime}^{(1)}$ has no exponential decaying. Therefore the total energy of the  Timoshenko system \eqref{timo}-\eqref{dirich} has no exponential decaying when $\frac{\rho_1}{k_1}\neq\frac{\rho_2}{k_2}$. The proof is thus complete.



\subsection{Polynomial stability in the general case}\label{se5}

\indent In this part, we prove that the system \eqref{eqq2.1}-\eqref{eqq2.3} is polynomially stable if \eqref{eqq3.1} is not satisfied.
We prove  the following Theorems.
\begin{thm}\label{TH4.2}
	Under hypothesis {\rm (H)},  if \begin{equation}\label{l4}
	\frac{\rho_1}{k_1}\neq\frac{\rho_2}{k_2} \ \textrm{ and } \ k_1\neq k_3,
	\end{equation}
	then there exists  $c>0$ such that for every $U^0\in D\left(\mathcal{A}\right)$, we have
	\begin{equation}\label{eqq4.40}
	E\left(t\right)\leq \frac{c}{\sqrt{t}}\left\|U^0\right\|^2_{D\left(\mathcal{A}\right)},\quad\ t>0.
	\end{equation}
\end{thm}
\begin{thm}\label{TH4.3}
	Under hypothesis {\rm (H)},  if
	\begin{equation}\label{l2}
	\frac{\rho_1}{k_1}\neq\frac{\rho_2}{k_2} \ \textrm{ and } \  k_1= k_3,
	\end{equation} then   there exists  $c>0$ such that for every $U^0\in D\left(\mathcal{A}\right)$, we have
	\begin{equation}\label{eqq4.41}
	E\left(t\right)\leq \frac{c}{t}\left\|U^0\right\|^2_{D\left(\mathcal{A}\right)},\quad\ t>0.
	\end{equation}
\end{thm}
Since $
\mathit{i}\mathbb{R}\subseteq\rho\left(\mathcal{A}\right),
$ then for the proof of Theorem \ref{TH4.2} and Theorem \ref{TH4.3}, according to \cite{borichev:10} (see also \cite{liu:05}), we still need to prove that
\begin{equation*}
\sup_{\lambda\in\mathbb{R}}\left\|\left(i\lambda Id-\mathcal{A}\right)^{-1}\right\|_{\mathcal{L}\left(\mathcal{H}\right)}=O\left(|\lambda|^l\right) \,\,\,\,\,\,\,\,  {\rm (H3)},
\end{equation*}
where $l=4$ if condition \eqref{l4} holds and $l=2$ if condition \eqref{l2} holds.
By a contradiction argument, suppose  there exist  a sequence of real numbers $\left(\lambda_n\right)_n$, with  $\lambda_n\to+\infty,$ and a sequence of vectors
\begin{equation}\label{eqq4.3}
U_n=\left(v^1_n,v^2_n,v_n^3,v^4_n,v^5_n,v^6_n,v^7_n\right)^{\mathsf{T}}\in D\left(\mathcal{A}\right) \ \text{ with }\ \|U_n\|_{\mathcal{H}}=1
\end{equation}
such that
\begin{equation}\label{eqq4.4}
\lambda_n^l\left(\mathit{i}\lambda_n U_n-\mathcal{A}U_n\right)=\left(f^1_n,f^2_n,f_n^3,f^4_n,f^5_n,f^6_n,f^7_n\right)^{\mathsf{T}}\to 0\ \text{ in } \mathcal{H}.
\end{equation}
Equivalently, we have 
\begin{eqnarray}
\mathit{i}\lambda_n v^1_n-v^4_n&=&h^1_n ,\label{eqq4.5}
\\
\mathit{i}\lambda_n v^2_n-v^5_n&=&h^2_n,\label{eqq4.6}
\\
\mathit{i}\lambda_n v^3_n-v^6_n&=&h^3_n, \label{eqq4.7}
\\
\rho_1\lambda_n^2 v^1_n+k_1 \left[\left(v^1_n\right)_x+v^2_n+\mathrm{l} v^3_n\right]_x+\mathrm{l} k_3\left[\left(v^3_n\right)_x-\mathrm{l} v^1_n\right]
&=&h^4_n,\label{eqq4.8}
\\
\rho_2\lambda_n^2 v^2_n   +\widetilde{k}_2 \left(v^2_n\right)_{xx}-k_1\left[\left(v^1_n\right)_x+v^2_n+\mathrm{l} v^3_n\right]+\int_0^{+\infty} g\left(s\right)\left(v^7_n\right)_{xx}ds&=&h^5_n,\label{eqq4.9}
\\
\rho_1\lambda_n^2 v^3_n+k_3\left[\left(v^3_n\right)_x-\mathrm{l} v^1_n\right]_x-\mathrm{l} k_1\left[\left(v^1_n\right)_x+v^2_n+\mathrm{l} v^3_n\right]&=&h^6_n ,\label{eqq4.10}
\\
\mathit{i}\lambda_n v^7_n +\left(v^7_n\right)_{s}-\mathit{i}\lambda_n v^2_n &=& h^7_n,\label{eqq4.11}
\end{eqnarray}
where
\begin{equation*}
\left\{
\begin{array}{ll}
\lambda_{n}^{l}h^1_n=f^1_n,\ \lambda_{n}^{l}h^2_n=f^2_n,\ \lambda_{n}^{l}h^3_n=f^3_n, \lambda_{n}^{l}h^7_n=f^7_n-f^2_n,\\ \\
\lambda_{n}^{l}h^4_n=-\rho_1 \left(
f^4_n+\mathit{i}\lambda_nf^1_n\right),\ \lambda_{n}^{l}h^5_n=-\rho_2\left(  f^5_n+\mathit{i}\lambda_nf^2_n\right),\
\lambda_{n}^{l}h^6_n=-\rho_1\left(  f^6_n+\mathit{i}\lambda_nf^3_n\right).
\end{array}
\right.
\end{equation*}
In the following we will check the condition {\rm (H3)} by finding a contradiction
with \eqref{eqq4.3} such as $\left\| U_n\right\|_\mathcal{H} =o(1)$. For clarity, we divide the proof into several lemmas.
From  now on, for simplicity, we drop the index $n$.
From \eqref{eqq4.5}-\eqref{eqq4.7}, we remark that
\begin{equation}\label{eqq4.13}
\left\|v^1\right\|=O\left(\frac{1}{\lambda}\right),\ \left\|v^2\right\|=O\left(\frac{1}{\lambda}\right),\ \left\|v^3\right\|=O\left(\frac{1}{\lambda}\right).
\end{equation}
Therefore, from \eqref{eqq4.8}-\eqref{eqq4.10}, we remark that
\begin{equation}\label{eqq4.15}
\left\| v^1_{xx}\right\|=O\left(\lambda\right),\ \left\|v^2_{xx}+\int_0^{+\infty}g\left(s\right)v^7_{xx}ds\right\|=O\left(\lambda\right),\ \left\|v^3_{xx}\right\|=O\left(\lambda\right).
\end{equation}
\begin{lem}\label{LE4.1} Let $l\geq0$. Under hypothesis  {\rm (H)}, we have
	\begin{equation}\label{eqq4.16}
	\int_0^{L}\int_0^{+\infty}g\left(s\right)\left|v^7_x\right|^2dsdx=o\left(\frac{1}{\lambda^l}\right).
	\end{equation}
\end{lem}
\begin{proof}
	Taking the inner product of \eqref{eqq4.4} with $U$ in $\mathcal{H}$. Then, using \eqref{eqq2.8} and the fact that $U$ is uniformly bounded in $\mathcal{H}$, we get
	\begin{equation}\label{eqq4.17'}
	\frac{1}{2}\int_0^{L}\int_0^{+\infty}g'\left(s\right)\left| v^7_x\right|^2dsdx=\Re\left(\left<\mathcal{A}U,U\right>_{\mathcal{H}}\right)=-\Re\left(\left<\mathit{i}\lambda U-\mathcal{A}U,U\right>_{\mathcal{H}}\right)=o\left(\frac{1}{\lambda^l}\right).
	\end{equation}
	Using condition {\rm (H)} in \eqref{eqq4.17'}, we get
	$$
	\int_0^{L}\int_0^{+\infty}g\left(s\right)\left|v^7_x\right|^2dsdx=o\left(\frac{1}{\lambda^l}\right).$$
	Thus the proof is complete.
\end{proof}
\begin{lem}\label{LE4.2'} Let $l\geq0$. Under hypothesis {\rm (H)},  we have
	\begin{equation}\label{eqq4.18}
	\intspace\left|v^2_x\right|^2dx=o\left(\frac{1}{\lambda^{\frac{l}{2}}}\right).
	\end{equation}
\end{lem}
\begin{proof} 
	Multiplying \eqref{eqq4.11} by $\overline{v^2}$ in $L^2_g\left(\mathbb{R}_+,H_0^1\right)$. Then,  using the fact that $\left\|v^2\right\|^2_{g}=g^0\left\| v^2_x\right\|^2$, $v^2_x$ is uniformly bounded in $L^2(0,L)$, $f^2$ converges to zero in $H^1_0\left(0,L\right)$ and $f^7$ converges to zero in $L^2_g\left(\mathbb{R}_+,H_0^1\right)$,    we get
	\begin{equation}\label{eqq4.22}
	g^0\lambda\intspace\left|v^2_x\right|^2dx=\lambda\intspace\inttime g\left(s\right)v^7_x\overline{v_x^2}ds dx-
	\mathit{i}\intspace\inttime g\left(s\right) v^7_{xs}\overline{v^2_x} ds dx+o\left(\frac{1}{\lambda^l}\right).
	\end{equation}
	From equation \eqref{eqq4.3} and Lemma \ref{LE4.1}, we get
	\begin{equation}\label{eqq4.22'}
	\lambda\intspace\inttime g\left(s\right)v^7_x\overline{v_x^2}ds dx=o\left(\frac{1}{\lambda^{\frac{l}{2}-1}}\right).
	\end{equation}
	Applying by parts integration,  Holder's inequality in $L^2(0,L)$ and $L^2(0,+\infty)$. Then, using \eqref{eqq4.17'}, the fact that $v^2_x$ is uniformly bounded in $L^2(0,L)$ and  $\displaystyle{\lim_{s\to0}\sqrt{g\left(s\right)}}$ exists, we get
	\begin{equation}\label{eqq4.22''}
	\left|\intspace\inttime g\left(s\right) v^7_{xs}\overline{v^2_x} ds dx\right|\leq  \displaystyle{\lim_{s\to0}\sqrt{g\left(s\right)}}\left(\int_0^{L}
	\int_0^{+\infty}-g'\left(s\right)
	\left|v^7_x\right|^2dsdx\right)^{1/2} \left\|v^2_x\right\|=o\left(\frac{1}{\lambda^{\frac{l}{2}}}\right).
	\end{equation}
	Inserting   \eqref{eqq4.22'} and \eqref{eqq4.22''}  into  \eqref{eqq4.22}, we deduce the  estimation of \eqref{eqq4.18}.  Thus the proof is complete.
\end{proof}
\begin{lem}\label{LE4.2} Let $l\geq0$. Under hypothesis {\rm (H)},  we have
	\begin{equation}\label{Eqq4.18}
	\intspace\left|v^2_x\right|^2 dx=o\left(\frac{1}{\lambda^l}\right).
	\end{equation}
\end{lem}
\begin{proof}
	Let $\displaystyle{l_N=\frac{l}{2}\sum_{k=0}^{N}\frac{1}{2^k}}$. Since $\displaystyle{\lim_{N\rightarrow+\infty}l_N}=l$, it is enough  to prove by induction on $N \in \mathbb{N}$ that
	\begin{equation}\label{eqq4.19}
	\int_0^{L}\left|v^2_x\right|^2dx=o\left(\frac{1}{\lambda^{l_N}}\right).
	\end{equation}
	When $N=0$, estimation \eqref{eqq4.19} holds by Lemma \ref{LE4.2'}. Suppose that  \eqref{eqq4.19} holds for $N-1$; i.e,
	\begin{equation}\label{eqq4.23}
	\lambda^{l_{N-1}}\int_0^{L}\left|v^2_x\right|^2dx=o\left(1\right).
	\end{equation}
	Multiplying  \eqref{eqq4.11} by $\lambda^{l_{N}}\overline{v^2}$
	in $L^2_g\left(\mathbb{R}_+,H_0^1\right)$.  Then  using the fact that $\left\|v^2\right\|^2_{g}=g^0\left\| v^2_x\right\|^2$, $v^2_x$ is uniformly bounded in $L^2(0,L)$, $f^2$ converges to zero in $H^1_0\left(0,L\right)$ and $f^7$ converges to zero in $L^2_g\left(\mathbb{R}_+,H_0^1\right)$,    we get  
	\begin{equation}\label{eqq4.24}
	\begin{array}{ll}
	\lambda^{l_{N}}g^0\intspace\left|v^2_x\right|^2dx=\intspace\inttime g\left(s\right)\lambda^{\frac{l}{2}} v^7_x\lambda^{l_{N}-\frac{l}{2}} \overline{v^2_x}dsdx\\ \\ \hspace{2.8cm}-\frac{\mathit{i}}{\lambda}\intspace\inttime g\left(s\right)\lambda^{\frac{l}{2}}v^7_{xs}\lambda^{l_{N}-\frac{l}{2}}  \overline{v^2_x}dsdx+o\left(\dfrac{1}{\lambda^{l+1-l_N}}\right).
	\end{array}
	\end{equation}
	Using the fact that $l_N-\frac{l}{2}= \frac{l_{N-1}}{2}$,  Lemma \ref{LE4.1},  \eqref{eqq4.22''} and  \eqref{eqq4.23}, we get
	$$
	\lambda^{l_{N}}\intspace\left|v^2_x\right|^2dx=o\left(1\right).
	$$
	Thus the proof is complete.
\end{proof}
\begin{lem}\label{LE4.3'}
	Let $2\leq l\leq4$. Under hypothesis {\rm (H)}, we have
	\begin{equation}\label{eqq4.29}
	\int_0^{L}\varsigma\left| v^1_x\right|^2dx=o\left(\frac{ 1}{ \lambda^{\frac{l}{2}-1}}\right),
	\end{equation}
	where  $\varsigma$ is the cut-off function defined in Section  \ref{se3}.
\end{lem}
\begin{proof} Since $l\geq2$, from Lemma \ref{LE4.1} and  Lemma  \ref{LE4.2} we have
	\begin{equation}\label{l>2}
	\intspace\left| v_x^2\right|^2=o\left(\frac{1}{\lambda^2}\right) \ \ \textrm{and} \ \ \intspace\inttime g\left(s\right)\left| v^7_x\right|^2dsdx=o\left(\frac{1}{\lambda^2}\right).
	\end{equation}
	Multiplying \eqref{eqq4.9} by $\varsigma  \overline{v^1_x}$ in $L^2\left(0,L\right)$. Then,  using  \eqref{eqq4.13} , \eqref{l>2},  and the fact that  $v^1_x$ is uniformly bounded in $L^2(0,L)$, $f^2$ converges to zero in $H^1_0\left(0,L\right)$,  $f^5$ converges to zero in $L^2\left(0,L\right)$,    we get we get
	\begin{equation}\label{eqq4.28}
	\int_0^{L}\varsigma\left| v^1_x\right|^2dx=o\left(1\right).
	\end{equation}
	Next,  multiplying \eqref{eqq4.9} by $\lambda^{\frac{l}{2}-1}\varsigma  \overline{v^1_x}$ in $L^2\left(0,L\right)$.  Then  using the fact that $v^1_x$ is uniformly bounded in $L^2(0,L)$, $f^2$ converges to zero in $H^1_0\left(0,L\right)$ and $f^5$ converges to zero in $L^2\left(0,L\right)$,    we get
	\begin{equation*}
	\begin{array}{l}
	\displaystyle{
		-k_1\lambda^{\frac{l}{2}-1}\intspace\varsigma\left| v^1_x\right|^2dx-\rho_2\intspace\varsigma  \lambda^{\frac{l}{2}} v^2_x\lambda\overline{v^1}dx-\rho_2\intspace\varsigma' \lambda^{\frac{l}{2}} v^2\lambda\overline{v^1}dx}\\ \\  \displaystyle{
		-\intspace\lambda^{\frac{l}{2}}\left(\widetilde{k}_2  v^2_x+\inttime g\left(s\right)v^7_{x}ds\right)\varsigma\lambda^{-1} \overline{v^1_{xx}}dx-\intspace\lambda^{\frac{l}{2}-1}\left(\widetilde{k}_2  v^2_x+\inttime g\left(s\right)v^7_{x}ds\right)\varsigma' \overline {v^1_{x}}dx}\\ \\  \displaystyle{-k_1\intspace\lambda^{\frac{l}{2}-1}\left( v^2+\mathrm{l} v^3\right)\varsigma \overline{v^1_{x}}dx}=o\left(\frac{1}{\lambda^{\frac{l}{2}}}\right).
	\end{array}
	\end{equation*}
	Since $l\leq4$, due to \eqref{eqq4.13}, we have
	$$
	\lambda^{\frac{l}{2}-1}\intspace\left( v^2+\mathrm{l} v^3\right)dx=O(1).
	$$
	From \eqref{eqq4.3}, \eqref{eqq4.4},  \eqref{eqq4.13}, \eqref{eqq4.28}, Lemma \ref{LE4.1}, and Lemma  \ref{LE4.2} we obtain
	\eqref{eqq4.29}. Thus the proof is complete.
\end{proof}
\begin{lem}\label{LE4.3}
	Let $2\leq l\leq 4$. Under hypothesis {\rm (H)}, we have
	\begin{equation}\label{eqq4.25}
	\int_0^{L}\varsigma\left| v^1_x\right|^2dx=o\left(\frac{1}{ \lambda^{l-2}}\right),\ \text{ and } \int^{L}_0\varsigma\left| v^1\right|^2dx=o\left(\frac{1}{\lambda^{l}}\right).
	\end{equation}
\end{lem}
\begin{proof}
	Let $\displaystyle{l_N=\frac{l-2}{2}\sum_{k=0}^{N}\frac{1}{2^k}}$. Since $ \displaystyle{\lim_{N\to+\infty}\frac{l-2}{2}\sum_{k=0}^N\frac{1}{2^k}=l-2}$,   we  prove by induction on $N \in \mathbb{N}$ that
	\begin{equation}\label{eqq4.26}
	\int_0^{L}\varsigma\left|v^1_x\right|^2dx=o\left(\frac{1}{\lambda^{l_N}}\right).
	\end{equation}
	If $N=0$, estimation \eqref{eqq4.26} holds by Lemma \ref{LE4.3'}. Suppose that
	\begin{equation}\label{eqq4.30}
	\lambda^{l_{N-1}}\int_0^{L}\varsigma\left| v^1_x\right|^2dx=o\left(1\right).
	\end{equation}
	Multiplying \eqref{eqq4.8} by $\lambda^{l_{N-1}}\varsigma\overline{ v^1}$ in $L^2\left(0,L\right).$  Then,  using the fact that $v^1$ is uniformly bounded in $L^2(0,L)$, $f^1$ converges to zero in $H^1_0\left(0,L\right)$ and $f^4$ converges to zero in $L^2\left(0,L\right)$,    we get
	\begin{equation*}
	\begin{array}{l}
	\displaystyle{
		\rho_1\lambda^{l_{N-1}+2} \intspace\varsigma\left|v^1\right|^2dx-k_1\lambda^{l_{N-1}}\intspace\varsigma\left|v^1_x\right|^2dx+k_1 \intspace\lambda^{\frac{l_{N-1}}{2}}\left(\varsigma v^2_x-\varsigma'v^1_x\right)\lambda^{\frac{l_{N-1}}{2}} \overline{v^1}dx}\\ \\
	\displaystyle{
		-\mathrm{l}\left(k_1+k_3\right)\intspace\lambda^{\frac{l_{N-1}}{2}}  v^3\lambda^{\frac{l_{N-1}}{2}}\left(\varsigma v^1\right)_xdx-\mathrm{l}k_3\lambda^{l_{N-1}} \intspace\varsigma\left|v^1\right|^2dx
		=o\left(\frac{1}{\lambda^{l-l_{N-1}}}\right)}.
	\end{array}
	\end{equation*}
	As $\frac{l_{N-1}}{2}\leq 1$, from  \eqref{eqq4.13}, \eqref{eqq4.30} and  Lemma \ref{LE4.2}, we obtain
	\begin{equation}\label{eqq4.32}
	\lambda^{l_{N-1}+2} \intspace\varsigma\left|v^1\right|^2dx=o\left(1\right).
	\end{equation}
	On the other hand, using \eqref{eqq4.32} and  Lemma \ref{LE4.2}, we get from \eqref{eqq4.8} that
	\begin{equation}\label{eqq4.33}
	\lambda^{-1+\frac{l_{N-1}}{2}}\intspace\varsigma\left|v^1_{xx}\right|^2dx=O\left(1\right).
	\end{equation}
	Multiplying  \eqref{eqq4.9} by $\lambda^{l_{N}}\varsigma\overline{ v_x^1}$ in $L^2\left(0,L\right).$  Then,  using the fact that $l_N=\frac{l-2}{2}+\frac{l_{N-1}}{2}$,   $v^1_x$ is uniformly bounded in $L^2(0,L)$, $f^2$ converges to zero in $H^1_0\left(0,L\right)$ and $f^5$ converges to zero in $L^2\left(0,L\right)$,    we get
	\begin{equation*}
	\begin{array}{l}
	\displaystyle{
		-\rho_2\intspace\lambda^{\frac{l}{2}} v^2_x\lambda^{1+\frac{l_{N-1}}{2}}\varsigma \overline{v^1}dx-\rho_2\intspace \lambda^{\frac{l}{2}} v^2\lambda^{1+\frac{l_{N-1}}{2}}\varsigma'\overline{v^1}dx-k_1\lambda^{l_{N}}\intspace\varsigma\left| v^1_x\right|^2dx}\\ \\  \displaystyle{
		-\intspace\lambda^{\frac{l}{2}}\left(\widetilde{k}_2  v^2_x+\inttime g\left(s\right)v^7_{x}ds\right)\lambda^{-1+\frac{\ell_{N-1}}{2}} \left(\varsigma \overline{v^1_{xx}}+\varsigma' \overline{v^1_x}\right)dx}\\ \\  \displaystyle{-k_1\intspace\lambda^{\frac{l}{2}-1}\left( v^2+l v^3\right)\lambda^{\frac{l_{N-1}}{2}}\varsigma \overline{v^1_{x}}dx}=o\left(\frac{1}{\lambda^{l-l_N}}\right).
	\end{array}
	\end{equation*}
	Using the fact that $2\leq l\leq 4$, \eqref{eqq4.4}, \eqref{eqq4.13}, \eqref{eqq4.30}, \eqref{eqq4.32}, \eqref{eqq4.33}, Lemma \ref{LE4.1}, and  Lemma \ref{LE4.2}, we get \eqref{eqq4.26}.  Therefore,
	\begin{equation}\label{eqq4.35}
	\intspace\varsigma\left| v^1_x\right|^2dx=o\left(\frac{1}{ \lambda^{l-2}}\right).
	\end{equation}
	Finally, multiplying \eqref{eqq4.8} by $\lambda^{l-2}\varsigma \overline{v^1}$ in $L^2\left(0,L\right)$. Then, using \eqref{eqq4.4},  \eqref{eqq4.13},    \eqref{eqq4.35}, and Lemma \ref{LE4.2}, we get  the second estimation  of \eqref{eqq4.25}.  Thus the proof is complete.
\end{proof}
\begin{lem}\label{LE4.4}
	Let $2\leq l\leq 4$. Under hypothesis {\rm (H)}, we have
	\begin{equation}\label{eqq4.41'}
	\mathrm{l}(k_1+k_3)\int_0^{L}\varsigma\left|v^3_x\right|^2dx+(k_3-k_1)\Re\left\{\intspace\lambda \varsigma v^1_{x} \lambda^{-1} \overline{v^3_{xx}}dx\right\}=o\left(1\right).
	\end{equation}
\end{lem}
\begin{proof} Multiplying  \eqref{eqq4.8}  by $\varsigma \overline{v^3_x}$ in $L^2\left(0,L\right)$. Then, using  \eqref{eqq4.13},  Lemma \ref{LE4.2}, Lemma  \ref{LE4.3},  $v^3_x$ is uniformly bounded in $L^2(0,L)$, $f^1$ converges to zero in $H^1_0\left(0,L\right)$ and $f^4$ converges to zero in $L^2\left(0,L\right)$,    we get
	\begin{equation}\label{eqq4.39}
	\rho_1\intspace\lambda^2 v^1\varsigma \overline{v^3_x}dx+\mathrm{l}\left( k_1+k_3\right)\int_0^{L}\varsigma\left|v^3_x\right|^2dx -k_1\intspace\lambda \varsigma v^1_{x} \lambda^{-1} \overline{v^3_{xx}}=o\left(1\right).
	\end{equation}
	Multiplying \eqref{eqq4.10} by $\varsigma\overline{ v^1_x}$ in $L^2\left(0,L\right)$. Then,  using  \eqref{eqq4.13},  
	Lemma \ref{LE4.3},  $v^1_x$ is uniformly bounded in $L^2(0,L)$, $f^3$ converges to zero in $H^1_0\left(0,L\right)$ and $f^6$ converges to zero in $L^2\left(0,L\right)$,    we get
	\begin{equation}\label{Eqq4.41}
	-\rho_1\intspace\lambda^2  \overline{v^1}\varsigma v^3_xdx+k_3\intspace\lambda\varsigma  \overline{v^1_x}\lambda^{-1}v^3_{xx}dx=o\left(1\right).
	\end{equation}
	Adding \eqref{eqq4.39} and \eqref{Eqq4.41}, then take the real part of the resulting equation, we get \eqref{eqq4.41'}. Thus the proof is complete.
\end{proof}

\noindent \textbf{Proof of Theorem \ref{TH4.2}} If  \eqref{l4} hold, take $\ell=4$ in Lemma \ref{LE4.1}, Lemma \ref{LE4.2}, and Lemma \ref{LE4.3}, we get
\begin{equation}\label{F01}
\intspace\inttime g\left(s\right)\left|v^7_x\right|^2ds dx=o\left(\frac{1}{\lambda^4}\right),\ \ \ 
\intspace\left|v^2_x\right|^2dx=o\left(\frac{1}{\lambda^4}\right).
\end{equation}
and 
\begin{equation}\label{F02}
\int_0^{L}\varsigma\left|v^1_x\right|^2dx=o\left(\frac{1}{\lambda^2}\right),\ \ \  \int^{L}_0\varsigma\left| v^1\right|^2dx=o\left(\frac{1}{\lambda^4}\right).
\end{equation}
Using \eqref{F02} and \eqref{eqq4.15}, we get 
\begin{equation}\label{aim'}
\intspace\lambda \varsigma v^1_{x} \lambda^{-1} \overline{v^3_{xx}}dx=o\left(1\right).
\end{equation}
From Lemma \ref{LE4.4} and \eqref{aim'}, we get
\begin{equation}\label{aim}
\int_0^{L}\varsigma\left|v^3_x\right|^2dx =o\left(1\right).
\end{equation}
Next, multiplying \eqref{eqq4.10} by $ \varsigma \overline{v^3}$ in $L^2\left(0,L\right).$  Then, using  
\eqref{F01}, \eqref{F02} , \eqref{aim}, $v^3$ is uniformly bounded in $L^2(0,L)$, $f^3$ converges to zero in $H^1_0\left(0,L\right)$ and $f^6$ converges to zero in $L^2\left(0,L\right)$,    we get 
\begin{equation}\label{F04}
\int_0^{L}\varsigma\left|v^3\right|^2dx =o\left(\frac{1}{\lambda^2}\right).
\end{equation}
Finally, using  \eqref{F01}, \eqref{F02},  \eqref{aim} and \eqref{F04}, we get
$\|U\|_{\mathcal{H}}=o\left(1\right),$ over $\left(\alpha+\epsilon,\beta-\epsilon\right)$. Then by applying Lemma \ref{LE3.3}, we deduce $\|U\|_{\mathcal{H}}=o\left(1\right),$ over $\left(0,L\right)$ which contradicts \eqref{eqq4.3}. This implies that $$\displaystyle{\sup_{\lambda\in\mathbb{R}}\left\|\left(i\lambda Id-\mathcal{A}\right)^{-1}\right\|_{\mathcal{L}\left(\mathcal{H}\right)}=O\left(\lambda^{4}\right)}.$$ The result follows from \cite{borichev:10}.\\

\noindent\textbf{Proof of  Theorem \ref{TH4.3}}
If  \eqref{l2} hold, take $\ell=2$ in  Lemma \ref{LE4.4}, we get directly 
\begin{equation}\label{Aim}
\int_0^{L}\varsigma\left|v^3_x\right|^2dx =o\left(1\right).
\end{equation}
Moreover, from Lemma \ref{LE4.1}, Lemma \ref{LE4.2}, and Lemma \ref{LE4.3}, we get
\begin{equation}\label{F001}
\intspace\inttime g\left(s\right)\left|v^7\right|^2 ds dx=o\left(\frac{1}{\lambda^2}\right),\ \ \ 
\intspace\left|v^2_x\right|^2dx=o\left(\frac{1}{\lambda^2}\right).
\end{equation}
and 
\begin{equation}\label{F002}
\int_0^{L}\varsigma\left|v^1_x\right|^2dx=o\left(1\right),\ \ \  \int^{L}_0\varsigma\left| v^1\right|^2dx=o\left(\frac{1}{\lambda^2}\right).
\end{equation}
Next, multiplying \eqref{eqq4.10} by $ \varsigma \overline{v^3}$ in $L^2\left(0,L\right).$  Then, using  
\eqref{Aim}, \eqref{F001}, \eqref{F002} , $v^3$ is uniformly bounded in $L^2(0,L)$, $f^3$ converges to zero in $H^1_0\left(0,L\right)$ and $f^6$ converges to zero in $L^2\left(0,L\right)$,    we get 
\begin{equation}\label{F004}
\int_0^{L}\varsigma\left|v^3\right|^2dx =o\left(\frac{1}{\lambda^2}\right).
\end{equation}
Finally, using \eqref{Aim}, \eqref{F001}, \eqref{F002}, and \eqref{F004}, we get 
$\|U\|_{\mathcal{H}}=o\left(1\right),$ over $\left(\alpha+\epsilon,\beta-\epsilon\right)$. Then by applying Lemma \ref{LE3.3}, we deduce $\|U\|_{\mathcal{H}}=o\left(1\right),$ over $\left(0,L\right)$ which contradicts \eqref{eqq4.3}. This implies that $$\displaystyle{\sup_{\lambda\in\mathbb{R}}\left\|\left(i\lambda Id-\mathcal{A}\right)^{-1}\right\|_{\mathcal{L}\left(\mathcal{H}\right)}=O\left(\lambda^{2}\right)}.$$ The result follows from \cite{borichev:10}.

\section{Thermo-elastic Bresse system with history and Cattaneo  law }\label{se6}

We can adapt similar analysis done in Section \ref{noheat}  to study the stability of the thermo-elastic Bresse system \eqref{eqq1.1'} with various boundary conditions given by \eqref{DDDD}, \eqref{DNDD} or \eqref{DNND}. In this section, we consider  system \eqref{eqq1.1'} with fully Dirichlet boundary conditions given by \eqref{DDDD} since the analysis of the stability of system \eqref{eqq1.1'} with the other boundary conditions follows easily.\\

\noindent After  introducing the new variable
\begin{equation*}
\begin{array}{lll}
\eta\left(x,t,s\right):=\psi\left(x,t\right)-\psi\left(x,t-s\right),&\text{in}& \left(0,L\right)\times\mathbb{R}_{+}\times\mathbb{R}_{+},
\end{array}
\end{equation*}
our system  \eqref{eqq1.1'}  takes the form
\begin{equation}\label{catan}
\left\{
\begin{array}{lll}
\displaystyle{
	\rho_1\varphi_{tt}-k_1 \left(\varphi_x+\psi+\mathrm{l} w\right)_x-\mathrm{l} k_3\left(w_x-\mathrm{l}\varphi\right)=0,}\\ \\
\displaystyle{\rho_2 \psi_{tt}-\left(k_2-\int_0^{+\infty}g\left(s\right)ds\right)\psi_{xx}+k_1\left(\varphi_x+\psi+\mathrm{l}w\right)-\int_0^{+\infty} g\left(s\right)\eta_{xx}ds+\delta\theta_x=0,}\\ \\
\displaystyle{\rho_1w_{tt}-k_3\left(w_x-\mathrm{l}\varphi\right)_x+\mathrm{l}k_1\left(\varphi_x+\psi+\mathrm{l}w\right)=0,}
\\ \\
\displaystyle{\eta_t+\eta_s-\psi_t=0,}\\ \\
\displaystyle{\rho_3 \theta_t+q_x+\delta \psi_{tx}=0,}\\ \\
\displaystyle{\tau q_t+\beta q+\theta_x=0,}
\end{array}
\right.
\end{equation}
with the  initial conditions
\begin{equation*}
\begin{array}{lll}
\varphi\left(\cdot,0\right)=\varphi_0\left(\cdot\right),\ \psi\left(\cdot,-t\right)=\psi_0\left(\cdot,t\right),\

w\left(\cdot,0\right)=w_0\left(\cdot\right),
\quad &\\
\varphi_t\left(\cdot,0\right)=\varphi_1\left(\cdot\right),\
\psi_t\left(\cdot,0\right)=\psi_1\left(\cdot\right),\
w_t\left(\cdot,0\right)=w_1\left(\cdot\right), \quad & \\
\theta\left(\cdot,0\right)=\theta_0\left(\cdot\right),\ q\left(\cdot,0\right)=q_0\left(\cdot\right),\quad & \\
\eta^0\left(\cdot,s\right):=\eta\left(\cdot,0,s\right)=\psi_0\left(\cdot,0\right)-\psi_0\left(\cdot,s\right),  & \text{in }(0,L), \ s\geq0,
\end{array}
\end{equation*}
and fully Dirichlet boundary conditions
\begin{equation*}
\begin{array}{lll}
\varphi\left(0,\cdot\right)=\varphi\left(L,\cdot\right)=\psi\left(0,\cdot\right)=\psi\left(L,\cdot\right)=0\quad&\text{in }\mathbb{R}_{+},\\
w\left(0,\cdot\right)=w\left(L,\cdot\right)=\theta\left(0,\cdot\right)=\theta\left(L,\cdot\right)=0\quad&\text{in }\mathbb{R}_{+},\\
\eta\left(0,\cdot,\cdot\right)=\eta\left(L,\cdot,\cdot\right)=0\quad&\text{in }\mathbb{R}_{+}\times\mathbb{R}_{+},\\
\eta\left(\cdot,\cdot,0\right)=0&\text{in} \left(0,L\right)\times\mathbb{R}_{+}.
\end{array}
\end{equation*}
We consider the energy space
\begin{equation*}
\mathcal{H}=\left(H_0^1\left(0,L\right)\right)^3\times\left(L^2\left(0,L\right)\right)^3\times L^2_g\left(\mathbb{R}_+,H^1_0\right)\times\left(L^2\left(0,L\right)\right)^2,
\end{equation*}
equipped with the norm
\begin{equation*}
\begin{array}{ll}
\displaystyle{\|U\|_{\mathcal{H}}^2}&= \displaystyle{\|\left(v^1,v^2,v^3,v^4,v^5,v^6,v^7,v^8,v^9\right)\|_{\mathcal{H}}^{2}}\\
&=\rho_1\left\|v^4\right\|^2+\rho_2\left\|v^5\right\|^2+\rho_1\left\|v^6\right\|^2+k_1\left\|v^1_x+v^2+\mathrm{l} v^3\right\|^2+\widetilde{k}_2\left\|v^2_x\right\|^2\\
&+k_3\left\|v^3_x-\mathrm{l} v^1\right\|^2+\left\|v^7\right\|^2_{g}+\rho_3\left\|v^8\right\|^2+\tau\left\|v^9\right\|^2.
\end{array}
\end{equation*}
Consider the linear unbounded operator $\mathcal{A}:D\left(\mathcal{A}\right)\to \mathcal{H}$ defined by
\begin{equation*}
\begin{array}{l}
D\left(\mathcal{A}\right)=\bigg\{\ U\in\mathcal{H} \ |\  v^1,v^3\in H^2\left(0,L\right),\  v^4,v^5,v^6\in  H^1_0\left(0,L\right),\\ \hspace{2cm} v^7_s\in L^2_g\left(\mathbb{R}_+,H^1_0\right),\   v^8\in H^1_0\left(0,L\right),\ v^9\in H^1\left(0,L\right),\\ \hspace{2cm} \ v^2+\int_0^{+\infty}g\left(s\right)v^7ds\in H^2\left(0,L\right)\cap H^1_0\left(0,L\right),\ v^7\left(x,0\right)=0\bigg\}  
\end{array}
\end{equation*}
and
\begin{equation*}
\mathcal{A}\left(\begin{array}{l}
v^1\\ v^2\\ v^3\\ v^4\\ v^5\\ v^6\\ v^7\\ v^8\\ v^9
\end{array}\right)=\left(\begin{array}{c}
v^4\\ v^5\\ v^6\\
\rho_1^{-1}\left(k_1 \left(v^1_x+v^2+\mathrm{l} v^3\right)_x+\mathrm{l} k_3\left(v^3_x-\mathrm{l} v^1\right)\right)\\
\rho_2^{-1} \left(\widetilde{k}_2 v^2_{xx}-k_1\left(v^1_x+v^2+\mathrm{l} v^3\right)+\int_0^{+\infty} g\left(s\right)v^7_{xx}ds-\delta v^8_x\right)\\
\rho_1^{-1}\left(k_3\left(v^3_x-\mathrm{l} v^1\right)_x-\mathrm{l} k_1\left(v^1_x+v^2+\mathrm{l} v^3\right)\right)\\
v^5-v^7_s\\ \rho^{-1}_3\left(-\delta v^5_x-v^9_x\right)\\
\tau^{-1}\left(-v^8_x-\beta v^9\right)
\end{array}\right),
\end{equation*}
for all $U=\left(v^1,v^2,v^3,v^4,v^5,v^6,v^7,v^8,v^9\right)^{\mathsf{T}}\in D\left(\mathcal{A}\right). $\\
Then system \eqref{eqq1.1'} is equivalent to the Cauchy problem
\begin{equation}\label{eqq2.7}
\left\{
\begin{array}{c}
U_t=\mathcal{A}U,\\
U\left(x,0\right)=U^0\left(x\right),
\end{array}
\right.
\end{equation}
where
$$
U=\left(\varphi,\psi,w,\varphi_t,\psi_t,w_t,\eta,\theta,q\right)^{\mathsf{T}}
$$
and $$
U^0\left(x\right)=\left(\varphi_0\left(x\right),\psi_0\left(x,0\right),w_0\left(x\right),
\varphi_1\left(x\right),\psi_1\left(x\right),w_1\left(x\right),\eta^0\left(x,.\right), \theta_0(x), q_0(x) \right)^{\mathsf{T}}.
$$
Note that $D(\mathcal{A})$ is dense in $\mathcal{H}$ and  that for all $U\in D\left(\mathcal{A}\right)$, we have
\begin{equation}\label{A7.3}
\left<\mathcal{A}U,U\right>_{\mathcal{H}}=\frac{1}{2}
\int_0^{L}\int_0^{+\infty}g'\left(s\right)\left| v^7_x\right|^2dsdx-\beta\int_0^{L}\left|v^9\right|^2dx.
\end{equation}
Consequently, under hypothesis {\rm(H)}, the system becomes dissipative. We can easily adapt the proofs in Subsection \ref{se2} and Subsection \ref{strong}
to prove the well-posedness and the strong stability of system \eqref{catan}. Furthermore,
similar to \cite{santos}, we define the following stability number
$$
\chi_0=\left(\tau-\frac{\rho_1}{\rho_3k_1}\right)\left(\rho_2-\frac{k_2\rho_1}{k_1}\right)-\frac{\tau\rho_1\delta^2}{\rho_3k_1}.
$$
\begin{thm}\label{expcatan}
	Under hypothesis  {\rm (H)},  if
	\begin{equation}\label{A7.4}
	\chi_0=0  \ \ \textrm{and} \ \ k_1=k_3,
	\end{equation}
	then system \eqref{catan} with fully Dirichlet boundary conditions is exponentially stable.
\end{thm}
\begin{proof}Similar to Theorem \ref{TH3.5}, we have to check  conditions {\rm (H1)} and {\rm (H2)}. We will prove condition (H2) by a contradiction argument.
	Suppose that there exists  a sequence of real numbers $\left(\lambda_n\right)_n$, with  $|\lambda_n|\to+\infty,$ and a sequence of vectors
	\begin{equation}\label{AA7.5}
	U_n=\left(v^1_n,v^2_n,v_n^3,v^4_n,v^5_n,v^6_n,v^7_n,v^8_n,v^9_n\right)^{\mathsf{T}}\in D\left(\mathcal{A}\right) \ \text{ with }\ \|U_n\|_{\mathcal{H}}=1
	\end{equation}  such that
	\begin{equation}\label{AA7.6}
	\mathit{i}\lambda_n U_n-\mathcal{A}U_n=\left(f^1_n,f^2_n,f_n^3,f^4_n,f^5_n,f^6_n,f^7_n,f^8_n,f^9_n\right)^{\mathsf{T}}\to 0\ \text{ in } \mathcal{H}.
	\end{equation}
	Equivalently, we have 
	\begin{eqnarray}
	\mathit{i}\lambda_n v^1_n-v^4_n&=&h^1_n,\label{AA7.7}
	\\
	\mathit{i}\lambda_n v^2_n-v^5_n&=&h^2_n,\label{AA7.8}
	\\
	\mathit{i}\lambda_n v^3_n-v^6_n&=&h^3_n,\label{AA7.9}
	\\
	\rho_1\lambda_n^2 v^1_n+k_1 \left[\left(v^1_n\right)_x+v^2_n+\mathrm{l} v^3_n\right]_x+\mathrm{l} k_3\left[\left(v^3_n\right)_x-\mathrm{l} v^1_n\right]
	&=& h^4_n,\label{AA7.10}
	\\
	\rho_2\lambda_n^2 v^2_n    +\widetilde{k}_2 \left(v^2_n\right)_{xx}-k_1\left[\left(v^1_n\right)_x+v^2_n+\mathrm{l} v^3_n\right]+\int_0^{+\infty} g\left(s\right)\left(v^7_n\right)_{xx} ds-\delta\left(v^8_n\right)_x&=& h^5_n,\label{AA7.11}
	\\
	\rho_1\lambda_n^2 v^3_n+k_3\left[\left(v^3_n\right)_x-\mathrm{l} v^1_n\right]_x-\mathrm{l} k_1\left[\left(v^1_n\right)_x+v^2_n+\mathrm{l} v^3_n\right]&=& h^6_n,\label{AA7.12}
	\\
	\mathit{i}\lambda_n v^7_n +\left(v^7_n\right)_{s}-\mathit{i}\lambda_n v^2_n &=&h^7_n,\label{AA7.13}\\
	\mathit{i}\rho_3\lambda_n v^8_n+\mathit{i}\delta\lambda_n \left( v^2_n\right)_x+\left(v^9_n\right)_x&=&h^8_n\label{AA7.14}\\
	\mathit{i}\tau\lambda_n v^9_n+\beta v^9_n+\left(v^8_n\right)_x&=&h^9_n,\label{AA7.15}
	\end{eqnarray}
	where
	\begin{equation*}
	\left\{
	\begin{array}{ll}
	\displaystyle{h^1_n=f^1_n,\ h^2_n=f^2_n,\ h^3_n=f^3_n,} \\ \\
	\displaystyle{h^4_n=-\rho_1\left( f^4_n+\mathit{i}\lambda_nf^1_n\right),\ h^5_n=-\rho_2\left(  f^5_n+\mathit{i}\lambda_nf^2_n\right),\
		h^6_n=-\rho_1 \left( f^6_n+\mathit{i}\lambda_nf^3_n\right),} \\ \\ 
	\displaystyle{ h^7_n=f^7_n-f^2_n,\ h^8_n=\rho_3 f^8_n+\delta\left(f^2_n\right)_x,\ h^9_n=\tau f^9_n.}
	\end{array}
	\right.
	\end{equation*}
	In the sequel, for shortness, we drop the index $n$. Taking the inner product of \eqref{AA7.6} with $U$ in $\mathcal{H}$. Then, using  \eqref{A7.3}, hypothesis  {\rm (H)} and the fact that $U$ is uniformly bounded in $\mathcal{H}$, we get 
	\begin{equation}\label{AA7.16}
	\intspace\inttime g\left(s\right) \left| v^7_x\right|^2ds dx=o\left(1\right)\ \  \text{ and }\ \ \intspace\left|v^9\right|^2dx=o\left(1\right).
	\end{equation}
	Similar to  Lemma \ref{LE3.2}, multiplying  \eqref{AA7.13}  by $\overline{v^2}$
	in $L^2_g\left(\mathbb{R}_+,H_0^1\right)$. Then, using \eqref{AA7.5} and \eqref{AA7.16},   we get
	\begin{equation}\label{AA7.17}
	\intspace\left|v^2_x\right|^2dx={o\left(1\right)}.
	\end{equation}
	Multiplying  \eqref{AA7.14} and \eqref{AA7.15}  by $\overline{v^8}$ and $\overline{v^9}$ respectively in $L^2\left(0,L\right)$. Then, using \eqref{AA7.6},  \eqref{AA7.16} and \eqref{AA7.17}, we get
	\begin{equation}\label{AA7.18}
	\intspace\left|v^8\right|^2dx=o\left(1\right).
	\end{equation}
	Multiplying \eqref{AA7.11} by $ \overline{v^2}$  in $L^2\left(0,L\right)$. Then, using \eqref{AA7.6} and \eqref{AA7.16}- \eqref{AA7.18}, we get
	\begin{equation}\label{NAA7.17}
	\intspace\left|\lambda v^2\right|^2dx=o\left(1\right).
	\end{equation}
	Multiplying  \eqref{AA7.10} and \eqref{AA7.11}  by  $\frac{\varsigma}{k_1}\left(\widetilde{k}_2\overline{v^2_{x}}+\int_0^{+\infty} g\left(s\right)\overline{v^7_{x}}ds\right)$ and $\varsigma \overline{v^1_x}$ respectively in $L^2\left(0,L\right)$. Then, take the real part of the resulting equation,  using \eqref{AA7.6} and  \eqref{AA7.16}-\eqref{AA7.18},  we get
	\begin{equation}\label{AA7.19}
	k_1\int_0^{L}\varsigma\left| v^1_x\right|^2dx+\lambda^2\left( \rho_2
	-\frac{\rho_1k_2}{k_1}\right)\Re\left\{\int_0^{L}\varsigma  v^2_x\overline{v^1}dx\right\}+\delta\Re\left\{\int_0^Lv^8_x\overline{v^1_x}dx\right\}=o\left(1\right),
	\end{equation}
	where  $\varsigma$ is the cut-off function defined in Subsection  \ref{se3}. Multiplying \eqref{AA7.10}, \eqref{AA7.14} , and \eqref{AA7.15}  by $\frac{\rho_3\tau}{\rho_1}\varsigma \overline{v^8}$ , $\mathit{i}\tau \lambda \varsigma \overline{v^1}$, and $\varsigma \overline{v^1_x}$  respectively in $L^2\left(0,L\right)$. Then, take the real part of the resulting equation,  using \eqref{AA7.6} and \eqref{AA7.16}-\eqref{AA7.18},  we get
	\begin{equation}\label{AA7.20}
	\lambda^2\Re\left\{\intspace\varsigma  v^2_x\overline{v^1}dx\right\}=-\frac{\rho_3 k_1}{\rho_1\delta\tau}\left(\tau-\frac{\rho_1}{\rho_3k_1}\right)\Re\left\{\intspace
	\varsigma  v^8_x\overline{v^1}dx\right\}+o\left(1\right).
	\end{equation}
	Inserting \eqref{AA7.20} in \eqref{AA7.19}, we get
	\begin{equation*}
	k_1\int_0^{L}\varsigma\left| v^1_x\right|^2dx-\frac{\rho_3k_1}{\rho_1\delta\tau}\chi_0 \Re\left\{\intspace v^8_x\overline{v^1_x}dx\right\}=o\left(1\right).
	\end{equation*}
	Using the fact that $\chi_0=0$, we get
	\begin{equation}\label{AA7.21}
	\int_0^{L}\varsigma\left| v^1_x\right|^2dx=o\left(1\right).
	\end{equation}
	Multiplying \eqref{AA7.10} by $\varsigma \overline{v^1}$  in $L^2\left(0,L\right)$. Then, using \eqref{AA7.6}, \eqref{AA7.17} and \eqref{AA7.21}, we get
	\begin{equation}\label{AA7.22}
	\int_0^{L}\varsigma\left|\lambda v^1\right|^2dx=o\left(1\right).
	\end{equation}
	Multiplying \eqref{AA7.10} and \eqref{AA7.12} by $\varsigma \overline{v^3_x}$   and $\varsigma \overline{v^1_x}$  respectively in $L^2\left(0,L\right)$. Then, take the real part of the resulting equation, using \eqref{AA7.6}, \eqref{AA7.16}-\eqref{AA7.17} and \eqref{AA7.21},  we get
	\begin{equation*}
	\mathrm{l}\left( k_1+k_3\right)\int_0^{L}\varsigma\left|v^3_x\right|^2dx +\left(k_3-k_1\right)\Re\left\{\intspace v^1_{x}\varsigma \overline{v^3_{xx}}dx\right\}=o\left(1\right).
	\end{equation*}
	Using the fact that $k_1=k_3$, we get
	\begin{equation}\label{AA7.23}
	\int_0^{L}\varsigma\left|v^3_x\right|^2dx=o\left(1\right).
	\end{equation}
	Moreover, multiplying \eqref{AA7.12} by $\varsigma \overline{v^3}$ in $L^2\left(0,L\right)$. Then, using \eqref{AA7.6}, \eqref{AA7.17} and \eqref{AA7.21}-\eqref{AA7.23}, we get
	\begin{equation}\label{AA7.24}
	\int_0^{L}\varsigma\left|\lambda v^3\right|^2dx=o\left(1\right).
	\end{equation}
	Finally, using \eqref{AA7.16}-\eqref{NAA7.17} and \eqref{AA7.21}-\eqref{AA7.24}, we can proceed similar to the proof of Theorem \ref{TH3.5} to get  the result of Theorem \ref{expcatan}.
\end{proof}
Note that when $\tau=0$, Cattaneo's law turns into Fourier law. In this case,  condition \eqref{A7.4} becomes equivalent to \eqref{eqq3.1}. However, if $\chi_0\neq0$ we can adapt the proof of Theorem \ref{TH4.2} and Theorem \ref{TH4.3} to show the following Theorems:

\begin{thm}\label{polycatan1}
	Under hypothesis {\rm (H)},  if \begin{equation}\label{conditionA}
	\chi_0\neq0\ \textrm{ and } \ k_1\neq k_3,
	\end{equation}
	then system \eqref{catan} with fully Dirichlet boundary conditions
	is polynomially stable with an energy  rate of decay $\displaystyle{\frac{1}{\sqrt{t}}}$, i.e,  there exists  $c>0$ such that for every $U^0\in D\left(\mathcal{A}\right)$, we have
	\begin{equation}\label{polA}
	E\left(t\right)\leq \frac{c}{\sqrt{t}}\left\|U^0\right\|^2_{D\left(\mathcal{A}\right)},\quad\ t>0.
	\end{equation}
\end{thm}
\begin{thm}\label{polycatan2}
	Under hypothesis {\rm (H)},  if
	\begin{equation}\label{conditionB}
	\chi_0\neq0 \ \textrm{ and } \  k_1= k_3,
	\end{equation} then system \eqref{catan} with fully Dirichlet boundary conditions
	is polynomially stable with an energy  rate of decay $\displaystyle{\frac{1}{t}}$, i.e,  there exists  $c>0$ such that for every $U^0\in D\left(\mathcal{A}\right)$, we have
	\begin{equation}\label{polB}
	E\left(t\right)\leq \frac{c}{t}\left\|U^0\right\|^2_{D\left(\mathcal{A}\right)},\quad\ t>0.
	\end{equation}
\end{thm}
Similar to Theorem \ref{TH4.2} and Theorem \ref{TH4.3}, we have to check   {\rm(H3)} where $l=4$ if condition \eqref{conditionA} holds and $l=2$ if condition \eqref{conditionB} holds. We will prove condition {\rm(H3)} by a contradiction argument, suppose  there exists  a sequence of real numbers $\left(\lambda_n\right)_n$, with  $\lambda_n\to+\infty,$ and a sequence of vectors
\begin{equation}\label{A7.6}
U_n=\left(v^1_n,v^2_n,v_n^3,v^4_n,v^5_n,v^6_n,v^7_n,v^8_n,v^9_n\right)^{\mathsf{T}}\in D\left(\mathcal{A}\right) \ \text{ with }\ \|U_n\|_{\mathcal{H}}=1
\end{equation}
such that
\begin{equation}\label{A7.7}
\lambda_n^l\left(\mathit{i}\lambda_n U_n-\mathcal{A}U_n\right)=\left(f^1_n,f^2_n,f_n^3,f^4_n,f^5_n,f^6_n,f^7_n,f^8_n,f^9_n\right)^{\mathsf{T}}\to 0\ \text{ in } \mathcal{H};
\end{equation}
Equivalently, we have 
\begin{eqnarray}
\mathit{i}\lambda_n v^1_n-v^4_n&=&h^1_n ,\label{A7.8}
\\
\mathit{i}\lambda_n v^2_n-v^5_n&=&h^2_n,\label{A7.9}
\\
\mathit{i}\lambda_n v^3_n-v^6_n&=&h^3_n, \label{A7.10}
\\
\rho_1\lambda_n^2 v^1_n+k_1 \left[\left(v^1_n\right)_x+v^2_n+\mathrm{l} v^3_n\right]_x+\mathrm{l} k_3\left[\left(v^3_n\right)_x-\mathrm{l} v^1_n\right]
&=&h^4_n,\label{A7.11}
\\
\rho_2\lambda_n^2 v^2_n   +\widetilde{k}_2 \left(v^2_n\right)_{xx}-k_1\left[\left(v^1_n\right)_x+v^2_n+\mathrm{l} v^3_n\right]+\int_0^{+\infty} g\left(s\right)\left(v^7_n\right)_{xx}ds-\delta\left(v^8_n\right)_x&=&h^5_n,\label{A7.12}
\\
\rho_1\lambda_n^2 v^3_n+k_3\left[\left(v^3_n\right)_x-\mathrm{l} v^1_n\right]_x-\mathrm{l} k_1\left[\left(v^1_n\right)_x+v^2_n+\mathrm{l}v^3_n\right]&=&h^6_n ,\label{A7.13}
\\
\mathit{i}\lambda_n v^7_n +\left(v^7_n\right)_{s}-\mathit{i}\lambda_n v^2_n &=& h^7_n,\label{A7.14}\\
\mathit{i}\rho_3\lambda_n v^8_n+\mathit{i}\delta\lambda_n \left( v^2_n\right)_x+\left(v^9_n\right)_x&=&h^8_n\label{A7.15}\\
\mathit{i}\tau\lambda_n v^9_n+\beta v^9_n+\left(v^8_n\right)_x&=&h^9_n\label{A7.16}.
\end{eqnarray}
where
\begin{equation*}
\left\{
\begin{array}{ll}
\displaystyle{\lambda_{n}^{l}h^1_n=f^1_n,\ \lambda_{n}^{l}h^2_n=f^2_n,\ \lambda_{n}^{l}h^3_n=f^3_n,}\\ \\ 
\displaystyle{\lambda_{n}^{l}h^4_n=-\rho_1 \left(
	f^4_n+\mathit{i}\lambda_nf^1_n\right),\ \lambda_{n}^{l}h^5_n=-\rho_2\left(  f^5_n+\mathit{i}\lambda_nf^2_n\right),\
	\lambda_{n}^{l}h^6_n=-\rho_1\left(  f^6_n+\mathit{i}\lambda_nf^3_n\right),} \\  \\ \displaystyle{\lambda_{n}^{l}h^7_n=f^7_n-f^2_n,\ \lambda_{n}^{l}h^8_n=\rho_3 f^8_n+\delta\left(f^2_n\right)_x,\ \lambda_{n}^{l}h^9_n=\tau f^9_n.}
\end{array}
\right.
\end{equation*}
In the sequel, for shortness, we drop the index n. Taking the inner product of \eqref{A7.7} with $U$ in $\mathcal{H}$. Then, using  \eqref{A7.3}, hypothesis  {\rm (H)} and the fact that $U$ is uniformly bounded in $\mathcal{H}$, we get 
\begin{equation}\label{catpol1}
\intspace\inttime g\left(s\right)\left|v^7_x\right|^2ds dx=o\left(\frac{1}{\lambda^l}\right)\ \ \text{ and } \ \ \intspace\left|v^9\right|^2dx=o\left(\frac{1}{\lambda^l}\right).
\end{equation}
Similar to Lemma \ref{LE4.2}, multiplying  \eqref{A7.14}  by $\overline{v^2}$
in $L^2_g\left(\mathbb{R}_+,H_0^1\right)$. Then, using \eqref{A7.7} and \eqref{catpol1},   we get
\begin{equation}\label{catpol2}
\intspace\left|v^2_x\right|^2dx=o\left(\frac{1}{\lambda^l}\right).
\end{equation}
From \eqref{A7.16} and \eqref{catpol1}, we get
\begin{equation}\label{catpol3}
\intspace\left|v^8_x\right|^2dx=o\left(\frac{1}{\lambda^{l-2}}\right).
\end{equation}
Similar to Lemma \ref{LE4.3'} and using \eqref{catpol3},  we can prove that for $2\leq l\leq 4$, we have
$$
\int_0^{L}\varsigma\left| v^1_x\right|^2dx=o\left(\frac{ 1}{ \lambda^{\frac{l}{2}-1}}\right),
$$
where  $\varsigma$ is the cut-off function defined in Subsection  \ref{se3}.
Consequently, for $2\leq l\leq 4$,  we can adapt the proof of Lemma \ref{LE4.3} to show that
\begin{equation}\label{catpol4}
\int_0^{L}\varsigma\left| v^1_x\right|^2dx=o\left(\frac{1}{ \lambda^{l-2}}\right)\ \text{ and } \int^{L}_0\varsigma\left| v^1\right|^2dx=o\left(\frac{1}{\lambda^{l}}\right).
\end{equation}
Finally, multiplying \eqref{A7.11} and \eqref{A7.13} by $\varsigma \overline{v^3_x}$   and $\varsigma \overline{v^1_x}$  respectively in $L^2\left(0,L\right)$. Then, take the real part of the resulting equation, using \eqref{A7.7}, \eqref{catpol1}-\eqref{catpol2} and \eqref{catpol4},  we get
\begin{equation}\label{A7.41}
\mathrm{l}(k_1+k_3)\int_0^{L}\varsigma\left|v^3_x\right|^2dx+(k_3-k_1)\Re\left\{\intspace\lambda \varsigma v^1_{x} \lambda^{-1}\overline{ v^3_{xx}}dx\right\}=o\left(1\right).
\end{equation}
\noindent\textbf{Proof of  Theorem \ref{polycatan1}} If \eqref{conditionA} hold,  we remark that  $l=4$ is the optimal value we can choose to get
\begin{equation}\label{A7.46}
\intspace\lambda \varsigma v^1_{x} \lambda^{-1}\overline{ v^3_{xx}}dx=o\left(1\right).
\end{equation}
Therefore, from \eqref{A7.41} and \eqref{A7.46}, we get
\begin{equation}\label{A7.45}
\int_0^{L}\varsigma\left|v^3_x\right|^2dx =o\left(1\right).
\end{equation}
Proceeding similar  to the proof of Theorem \ref{TH4.2}, we get  the result of Theorem \ref{polycatan1}.\\

\noindent\textbf{Proof of  Theorem \ref{polycatan1}} If \eqref{conditionB} hold,  then \eqref{A7.41} yields directly \eqref{A7.45}. In this case, we choose  $l=2$ as the optimal value of $2\leq l\leq4$. Proceeding similar  to the proof of Theorem \ref{TH4.3}, we get  the result of Theorem \ref{polycatan2}.
\begin{rk}
	Following Theorem 4.1 in \cite{Wehbenadine} the energy of the Bresse system with fully Dirichlet or mixed boundary conditions decays as $\dfrac{1}{\sqrt{t}}$ if only one thermal dissipation given by Fourier law is considered and $\dfrac{1}{\sqrt[3]{t}}$ if only one thermal dissipation given by Cattaneo law is considered. 
\end{rk}
\section{ Conclusion and open questions}
Bresse system \eqref{eqq1.1'} with dissipative thermal effect given by Cattaneo's law and history type control is expected to decay faster than system \eqref{eqq1.1} without heat conduction. Nevertheless, Theorem \ref{TH4.2} and Theorem \ref{polycatan1} show that the heat dissipation does not affect the rate of energy decay. Consequently, the optimality of the polynomial decay rate of  system \eqref{eqq1.1'} and the influence of the Cattaneo law on the stability of the system remain an open problem.




\end{document}